\documentclass[reqno,11pt]{amsart}
\usepackage{amssymb}
\usepackage{amsmath,amsthm,amscd}
\newtheorem{Proposition}{Proposition}[section]
\newtheorem{Definition}[Proposition]{Definition}
\newtheorem{Lemma}[Proposition]{Lemma}
\newtheorem{Theorem}[Proposition]{Theorem}
\newtheorem{Corollary}[Proposition]{Corollary}
\newcommand \fourier{\widehat}

\newcommand \valg{\Val^G}
\newcommand \valun{\Val^{U(n)}}
\newcommand \valuinf{\Val^{U(\infty)}}
\newcommand \valsm{\Val^{sm}}
\newcommand \valsmplus{\Val^{sm,+}}
\newcommand \sltwo{\mathfrak{sl}(2)}

\newcommand\length{\operatorname{length}}

\DeclareMathOperator{\vol}{vol}

\newcommand{\norm}[1]{\left\| #1 \right\|}
\newcommand{\R}{\mathbb{R}}
\newcommand{\C}{\mathbb{C}}
\newcommand{\K}{\mathcal{K}}
\newcommand{\Ksm}{\mathcal{K}^{sm}}
\newcommand{\D}{\mathbb{D}}
\newcommand{\F}{\mathbb{F}}

\newcommand{\eps}{\epsilon}

\DeclareMathOperator{\kl}{Kl}
\DeclareMathOperator{\Val}{Val}
\DeclareMathOperator{\Gr}{Gr}
\DeclareMathOperator{\AGr}{\overline{ Gr}}

\title{Hermitian integral geometry}
\author{Andreas Bernig}
\author{Joseph H. G. Fu}
 
\email{bernig@math.uni-frankfurt.de}
\email{fu@math.uga.edu}
\date{\today}

\address{Institut f\"ur Mathematik, Goethe-Universit\"at Frankfurt,
Robert-Mayer-Str. 10, 60054 Frankfurt, Germany}
\address{ Department of Mathematics, 
University of Georgia, 
Athens, GA 30602, USA}

\begin{document}

\begin{abstract} We give in explicit form the principal kinematic formula for the action of the affine unitary group on $\C^n$, together with a straightforward algebraic method for computing the full array of unitary kinematic formulas, expressed in terms of certain convex valuations introduced, essentially, by H. Tasaki. We introduce also several other canonical bases for the algebra of unitary-invariant valuations, explore their interrelations, and characterize in these terms the cones of positive and monotone elements.

\end{abstract}

\thanks{ 
\\ Supported
  by the Schweizerischer Nationalfonds grants SNF PP002-114715 and SNF 200020-121506.}
\maketitle 

\section{Introduction}
\subsection{General background} In \cite{fu90}, it was shown that if $\overline G$ is a Lie group acting transitively on the sphere bundle of a Riemannian manifold $M$ then there exist kinematic formulas (cf. \eqref{preliminary kf} below) for certain geometric quantities associated to subspaces $A,B \subset M$; the case $\overline G= SO(n)\ltimes \R^n, M = \R^n$ being the classical kinematic formulas of Blaschke-Santal\'o-Federer-Chern. 
The proof was a distillation of the geometric method used in \cite{chern66} and \cite{federer59} to establish the classical case.

A different, and in some ways more incisive proof of the classical case was provided by \cite{hadwiger57}. Restricting formally to the case where the subspaces are convex sets, Hadwiger displayed a concrete finite basis for the vector space of continuous convex valuations invariant under the euclidean group. The existence of the kinematic formulas is then a simple consequence, and the precise numerical values of the coefficients involved may be calculated using the ``template method", i.e. by evaluating the relevant integrals for enough conveniently chosen $A,B$. This approach leaves the impression that the values of the coefficients are in some way accidental. However, A. Nijenhuis \cite{nijenhuis74} showed by direct calculation that under a suitable renormalization of the Hadwiger basis all of the coefficients are equal to unity.

More recently, S. Alesker \cite{ale03b} gave another proof of the theorem of \cite{fu90} as part of a far-reaching reconceptualization of the theory of convex valuations. He showed that if $G$ is a compact Lie group acting transitively on the sphere of a euclidean space $V$ then the space $\valg(V)$ of continuous convex valuations invariant under the group $\overline G := G\ltimes V$, generated by translations and the action of $G$, is finite dimensional. Just as in the case of the full euclidean group, the theorem of \cite{fu90} follows directly (at least in the euclidean case). In these terms, the result may be stated as follows. 
\begin{Theorem} Let $\phi_1\dots,\phi_N $ be a basis for $\valg(V)$. Given $\mu \in \valg(V)$, there are constants $c^\mu_{ij}, 1\le i,j\le N$, such that for any two compact convex bodies $A,B \subset V$
\begin{equation}\label{preliminary kf}
\int_{\overline G} \mu(A \cap \bar gB) \, d\bar g = \sum_{ij} c^\mu_{ij} \phi_i(A) \phi_j(B).
\end{equation}
\end{Theorem}

Moreover, in \cite{ale03b} Alesker gave an explicit basis (in fact two of them) for the space $\valun(\C^n)$ of unitary-invariant valuations on $\C^n$. Although this in itself gives a lot of information about the kinematic formulas, a complete determination of the formulas using the template method appears intractable (although H. Park \cite{pa02} used it successfully in the cases $n= 2,3$).

Meanwhile, H. Tasaki \cite{tasaki00, tasaki03}, building on previous work of R. Howard \cite{howard93}, established a more detailed description of the unitary kinematic formula, which he stated in the restricted  case where $A,B$ are compact submanifolds of complementary dimension. He showed first of all that if $k\le n$ then the unitary orbits of the (real) dimension $k$ (resp. codimension $k$) Grassmannian $\Gr_k(\C^n)$ (resp. $\Gr_{2n-k}(\C^n)$) are naturally parametrized by the $p:=\lfloor \frac k 2\rfloor$-simplex
$\{(\theta_1,\dots,\theta_{p}): 0\le \theta_1\le \dots\le \theta_{p}\le \frac \pi 2\}$. Put $\Theta(E)$ for this ``multiple K\"ahler angle" of $E \in \Gr_{k \text{ or } 2n-k}(\C^n)$, and $\cos^2 \Theta(E)$ for the vector with components $\cos^2\theta_i$. Tasaki's theorem may then be restated as follows.
\begin{Theorem}[Tasaki \cite{tasaki03}] Given $k\le n$, there is a symmetric $(p+1)\times( p+1)$ matrix $T= T^n_k$ such that whenever $A^k, B^{2n-k}\subset \C^n$ are $C^1$ submanifolds of dimension and codimension $k$ respectively,
\begin{equation}
\int_{\overline{U(n)}} \#(A \cap \bar g B) \, d\bar g = \sum_{ij} T_{ij} \int_A \sigma_i (\cos^2 \Theta(T_x A)) \, dx \int_B \sigma_j (\cos^2 \Theta(T_y B)) \, dy
\end{equation}
where $\sigma_i$ is the $i$th elementary symmetric function and $\overline{U(n)}= U(n) \ltimes \C^n$ is the affine unitary group.
\end{Theorem}

As Tasaki noted, this formula also holds verbatim if $\C^n$ is replaced by either of the complex space forms $\C P^n, \C H^n$, with their full groups of isometries. This is an instance of the {\it transfer principle}, which we discuss in section \ref{transfer principle section} below.

\subsection{Results of the present paper} In the pages to follow we bring more of the algebraic machinery introduced by Alesker  to bear on the problem of the integral geometry of the unitary group. The key underlying observation (Theorem \ref{thm:abstract} below) is that the graded multiplication introduced in \cite{ale04} on the space of convex valuations is intimately related to the various $G$-kinematic formulas. This illuminates even the classical $SO(n)$ theory, explaining the result of Nijenhuis cited above (cf. \cite{fu06}, section 2.3, and also \cite{fu06b}). 
Our point of entry is the determination in \cite{fu06} of the multiplicative structure of $\valun(\C^n)$.
Here we give a more or less complete set of answers to the questions posed in section 4 of \cite{fu06}.
We now describe our present results as they relate to those questions, in the order given there.

\begin{enumerate}
\item {\bf Explicit kinematic formulas for $U(n)$.}
The paper \cite{fu06} posed the problem of computing the kinematic formulas explicitly in terms of the {\bf monomial basis} (cf. section \ref{monomial basis}). This boils down to computing the inverses $Q^n_k$ of certain symmetric matrices $P^n_k$. It turns out that the $P^n_k$ are Hankel matrices with ascending entries of the form $\binom {2i} i$. Thus the expansion of the inverse as a polynomial in the matrix entries gives some kind of answer to this question, but it seems unreasonably complicated. It would be interesting to have a closed form. 

In the present paper we take a different approach, showing how to determine completely the unitary kinematic formulas (cf. Theorem \ref{pkf}, Corollary \ref{tasaki entries} and section \ref{kinematic formulas}) in terms of the {\bf Tasaki basis} (cf. Prop. \ref{tasaki vals} below) for $\valun$, obtaining in this way the {\bf Tasaki matrices} $T^n_k$, which may be obtained in principle by a change of basis from the $Q^n_k$. Although the formulas remain complicated, they are an order of magnitude less so than the na\"ive formulas for $Q^n_k$ described above. Using this approach we can show, for instance, that the $Q^n_k$ are positive definite (Corollary \ref{+ def cor}), answering another question of \cite{fu06}.
Furthermore the Tasaki valuations are more amenable to calculation in concrete geometric situations.
Strictly speaking we carry this out in full detail only for the principal kinematic formula $k_{U(n)}(\chi)$ (cf. \eqref{def k G} below), then show how the general formulas may be computed in an essentially straightforward way.

Among the many special bases for $\valun$ (cf. the next item) the Tasaki valuations seem to enjoy a privileged status. For example if $k=2p$ is even then, in addition to the usual diagonal symmetry  $ \left(T^n_k\right)_{ij}= \left(T^n_k\right)_{ji}$, the Tasaki matrices $T^n_{2p}$ display the unexpected antidiagonal symmetry
$ \left(T^n_{2p}\right)_{ij}= \left(T^n_{2p}\right)_{p-i,p-j}$ (Theorem \ref{palindrome thm}).

\item {\bf Canonical bases and their interrelations.} We explore with varying degrees of depth several canonical bases for $\valun(\C^n)$: the monomial basis and its Fourier transform, the hermitian intrinsic volumes  $\mu_{k,q}$ (which correspond in a natural way with certain differential forms on the tangent bundle of $\C^n$), their ``Crofton duals" $\nu_{k,q}$, and the Tasaki valuations $\tau_{k,q}$ and their Fourier transforms $\widehat{\tau_{k,q}}$. Although we explicitly study their interrelations only to the extent necessary to answer our other concerns, there is enough information here to give a complete (though again complicated) dictionary among them.

\item {\bf Special cones.} We show that the cone $P$ of nonnegative elements of $\valun$ is generated by the hermitian intrinsic volumes. Stimulated by the fact (due to Kazarnovskii) that the ``Kazarnovskii pseudo-volume" $\mu_{n,0}$ is at once nonnegative and non-monotone, we give a complete characterization of the cone $M$ of monotone elements of $\valun$.

\end{enumerate}

As a concluding general remark, we have taken care to give precise and complete values whenever possible. Beyond the obvious motive of providing solid information for possible applications, we mean to make the point (in the only way possible)  that this algebraic approach is sufficient to formulate these results in complete detail, in an area historically plagued by statements of the form ``There exists a formula such that...."

In the latter respect, however, things are not yet in a satisfactory state. Some results are given in terms of sums for which we have not found closed forms. Whether or not such closed forms exist, their nature suggests that there might exist some combinatorial model that generates them, perhaps something like the devices that occur in Schubert calculus. Indeed much of the approach in the following pages is inspired by the principle that the algebras $\valg (V)$ are similar to the cohomology algebras of K\"ahler manifolds--- it is even the case that the main subject of this paper, $\valun(\C^n)$, has the same Betti numbers as the even-degree cohomology of the Grassmann manifold of complex $2$-planes in $\C^{n+2}$, although the algebras themselves are not isomorphic.

\subsubsection{Acknowledgements} We wish to thank  Semyon Alesker, Ludwig Br\"ocker, Dan Nakano, Jason Parsley, Ted Shifrin and Robert Varley for helpful discussions, and the Universities of Georgia (USA) and Fribourg (Switzerland) for hosting our mutual visits as we worked out this material. We thank also the anonymous referee, whose many useful remarks greatly improved the text at key points, and who in particular suggested the proof of Theorem \ref{monotone homogeneous} in the non-smooth case. 


\section{Valuations and curvature measures} 
Throughout most of this section we let $V$ be an oriented euclidean vector space of dimension $n<\infty$. We note, however, that for much of the discussion the euclidean assumption is not strictly necessary if we substitute the dual space $V^*$ for $V$ in appropriate spots.

We put 
$$ \omega_k:= \frac{\pi^{\frac k 2}}{\Gamma(\frac k 2 +1)}$$
for the volume of the $k$-dimensional euclidean ball. In particular
$$\omega_{2l} = \frac{\pi^l}{l!},$$
which also happens to be the volume of the complex projective space $\C P^l$ under the Fubini-Study metric.

\subsection{Basics}
For definiteness we will work formally in the context of convex valuations on $V$. However, many statements apply also to other geometrically valid subsets (e.g. $C^2$ submanifolds, or in the case of the Crofton formulas even $C^1$ submanifolds) of smooth manifolds, in terms of the formalism of valuations on manifolds \cite{ale05a, ale05b,alefu05,ale05d,ale06}. Since these notions intervene only at the stage of interpretation of our main results, and never in an essential technical way, we will say no more about them.

We put $\K =\K(V)$ for the space of all compact convex subsets of $V$, endowed with the Hausdorff metric, and $\Ksm(V)\subset \K(V)$ for the subspace consisting of subsets with nonempty interior and smooth boundary, and for which all principal curvatures are nonzero. We refer to \cite{befu06} and the sources cited there for the definition and basic properties of the vector space $\Val = \Val(V)$ of continuous translation-invariant valuations on $V$, and of the dense subspace $\Val^{sm}(V)$ of smooth valuations.  Basic examples of these objects include the Euler characteristic $\chi$ and the volume measure $\vol_n$.

Recall that a valuation $\phi$ is of {\bf degree} $k$ if $\phi(tK)=t^k\phi(K)$ for all $t \geq 0$ and {\bf even} if $\phi(-K) = \phi(K)$ for all $K\in \K$. The corresponding subspace of $\Val$ is denoted by $\Val_k^+$. 
It is known \cite{kl00} that the restriction of an even valuation $\mu$ of degree $k$ to a $k$-dimensional subspace $E \subset V$ is a multiple of the restriction of the usual Hausdorff measure $\vol_k$ to $E$. Putting $\kl_{\mu}(E)$ for the proportionality factor, we obtain the {\bf Klain function} $\kl_\mu \in C(\Gr_k(V))$ of $\mu$. In other words,  $\kl_\mu$ is uniquely characterized by the relation
\begin{equation}
\mu(K) = \kl_\mu(E)\vol_k(K) \text{ for } E \in \Gr_k, \quad K \in \K( E).
\end{equation}
A theorem of Klain \cite{kl00} states that the resulting map $\kl$ from the space of even valuations of degree $k$ to $C(\Gr_k(V))$ is injective. 
%

Every even $\mu \in \valsm_k(V)$ admits a {\bf Crofton measure}, i.e. a signed measure $m$ on $\Gr_{k}(V)$ such that 
$$\mu (A) = \int_{\Gr_{k}(V)} \vol_k(\pi_E(A))\, dm(E), $$
where $\pi_E$ is the orthogonal projection to $E$. This follows from the Alesker-Bernstein theorem \cite{ale-be} (compare also Section 1 in [1]).

We recall also Alesker's {\bf Fourier transform} $\F:\valsm_k(V) \to\valsm_{n-k}(V) $ (cf. \cite{ale07}). 
In the present paper we will denote the Fourier transform of a valuation $\phi $ by 
\begin{equation}
\widehat \phi:= \F \phi.
\end{equation}
We will only make use of it for even valuations, in which case it is uniquely characterized in terms of the Klain embedding by
\begin{equation}
\kl_{\widehat \phi}(E)= \kl_\phi (E^\perp), \quad E \in \Gr_{n-k}(V)
\end{equation}
for even $\phi\in \valsm_k$.
In this form, the Alesker-Fourier transform was denoted by $\D$ in \cite{ale03b}, \cite{befu06} and in several other papers. 

Alesker has defined in \cite{ale04} a commutative graded product on $\valsm(V)$, with the property that the symmetric pairing 
\begin{equation}\label{pd pairing}
(\phi,\psi):= \text{ degree $n$ part of }\phi\cdot \psi
\end{equation}
 is perfect. 
 We recall \cite{befu06} that the related pairing 
\begin{equation}\label{fourier pairing}
\langle\phi,\psi\rangle:= (\phi,\fourier \psi)
\end{equation}
is symmetric. 
We will see later that the restriction of the pairing $\langle \cdot,\cdot \rangle$ to $\valun$ is positive definite. However this is not true of the unrestricted pairing--- it is shown in \cite{be07} that if $n$ is odd then the index of the restriction of the pairing to $ \Val^{SU(n)}(\C^n)$ is 1.

\subsection{Grassmannians} We denote the Grassmannian of $k$-dimensional subspaces of the real vector space $V$ by $\Gr_k(V)$.
If $V =\C^n$ (considered as a {\it real} vector space) we consider the {\bf $(k,p)$-Grassmannian} $\Gr_{k,p}(\C^n)\subset \Gr_k(\C^n)$ to be the submanifold of all $k$-dimensional real subspaces that may be expressed as the orthogonal direct sum of a $p$-dimensional complex subspace and a $(k-2p)$-dimensional real subspace that is isotropic with respect to the standard symplectic (K\"ahler) structure on $\C^n$. A general element of $\Gr_{k,p}$ will be denoted $E^{k,p}$. It is easy to see that $\Gr_{k,p}$ is the orbit of $\C^p \oplus \R^{k-2p}$ under the standard action of $U(n)$. In particular, $\Gr_{2p,p}$ is the Grassmannian of $p$-dimensional complex subspaces and $\Gr_{n,0}(\C^n)$ is the Lagrangian Grassmannian.

\subsection{Global and local} \label{globloc} We recall that the family of algebras $\valun:= \valun(\C^n)$ is related by the sequence of surjective restriction homomorphisms $\valun \to \Val^{U(n-1)}, n\ge 1$. The algebra $\valuinf$ of {\bf global valuations} is the inverse limit of this system; abusing terminology we will identify a global valuation with its images in the various $\valun$. An expression for an element of $\valun$ that does not hold in $\valuinf$ will be called {\bf local}, or {\bf local at $n$}. Likewise we will refer to global and local relations among valuations.

\subsection{Poincar\'e duality and kinematic formulas} We recall from \cite{ale03b} that if $G\subset O(V)$ is a compact group acting transitively on the sphere of $V$ then the space $\valg(V)$ of $G$-invariant and translation invariant valuations on $V$ is finite dimensional. It follows (cf. \cite{befu06}) that there is a linear injection $k_G: \valg\to\valg \otimes \valg$ such that whenever $A,B \in \K$ and $\phi \in \valg$
\begin{equation}\label{def k G}
k_G(\phi)(A,B)=\int_{\overline G} \phi(A\cap \bar g B)\, d\bar g .
\end{equation}
Here $\overline G:= G\ltimes V$ is the group generated by $G$ and the translation group of $V$ and $d\bar g$ is the Haar measure, normalized so that
\begin{equation}\label{standard convention}
d\bar g\left(\{\bar g: \bar g o \in S\}\right) = \vol_n(S), \quad S\subset V \text{ measurable,}
\end{equation}
where $o \in V$ is an arbitrarily chosen point.
 If $\phi \in\valg_k$, then $
k_G(\phi) \in \bigoplus_{i+j = n+k} \valg_i \otimes \valg_j$. Taking $A$ to be a point, it is clear that the term of $k_G(\chi)$ of bidegree $(0,n)$ is $\chi \otimes \vol_n$.

 The algebraic approach to the kinematic formula is based on the following statement from \cite{befu06}.  Let $p: \valg\to {\valg}^*$ denote the linear isomorphism induced by the Poincar\'e duality pairing \eqref{pd pairing},  $m_G:\valg\otimes\valg \to \valg$ the restriction of the multiplication map to $\valg$, and $m_G^*:{\valg}^*\to {\valg}^* \otimes{\valg}^*$ its adjoint.

 \begin{Theorem}\label{thm:abstract}
\begin{equation}\label{eq:abstract}(p\otimes p )\circ k_G = m_G^* \circ p\end{equation}
\end{Theorem}

To state this in more sensible terms:

\begin{Theorem}\label{g-kinematic}  Let $\phi_1,\dots,\phi_N$ and $\psi_1,\dots,\psi_N$ be bases of $\Val^G$, and let $M$ be the $N \times N$ matrix 
$$ M_{ij} := (\phi_i,\psi_j),$$
where the right hand side is given by the Poincar\'e duality pairing \eqref{pd pairing}. Let $K:= M^{-1}$. Then 
\begin{equation}
k_G(\chi) = \sum_{i,j} K_{ij}\, \phi_i \otimes \psi_j.
\end{equation}
If the $\psi_i = \fourier{\phi_i}$ then $M$ and $K$ are symmetric.
More generally, for any $\mu \in \valg$,
\begin{equation}\label{general kf}
k_G(\mu) = \sum_{i,j} K_{ij}\, (\mu\cdot\phi_i) \otimes \psi_j = \sum_{i,j} K_{ij}\, \phi_i \otimes (\mu\cdot\psi_j).
\end{equation}
\end{Theorem}

The symmetry assertion is of course the same as the symmetry of the pairing \eqref{fourier pairing}.

These formulas also apply to other types of geometric subsets of $V$, as described in \cite{fu90}, \cite{howard93}, \cite{fu94}. The simplest case occurs when $A,B$ are smooth compact submanifolds of complementary dimensions $k,n-k$. It is advantageous to use bases for $\valg$ comprised of bases for the components $\valg_k$ of the grading by degree. Given an even valuation $\phi \in \valg_k$, and a compact $C^1$ $k$-dimensional submanifold $A\subset V$, it is natural to put
$$ 
\phi(A):= \int_A \kl_\phi(T_xA) \, dx$$
and the kinematic formula yields the Crofton formula
\begin{equation}\label{poincare formula}
\int_{\overline G} \#(A \cap \bar gB) \, d\bar g = \sum_{\deg \phi_i = k,\deg \psi_j = n-k}K_{ij}\, \phi_i(A)\psi_j(B),
\end{equation}
where $\#$ denotes the cardinality.

\subsection{The transfer principle for Crofton formulas}\label{transfer principle section} R. Howard has established a general Crofton formula for Riemannian homogeneous spaces $M:=G/K$. Put $\overline{\Gr_m}(M)$ for the dimension $m$ Grassmann bundle over $M$.

\begin{Theorem}[\cite{howard93}] Let $G$ be a unimodular Lie group and $M:=G/K$ a Riemannian homogeneous space of $G$, and let $m+n \ge \dim M$.  Let the Haar measure on $M$ be given by \eqref{standard convention}. Then there exists a nonnegative function $f_{M,G,K}\in C^\infty( \overline{\Gr_m}(M) \times \overline{\Gr_n}(M)) $, invariant under the action of $G\times G$ on $\overline{\Gr_m}(M) \times \overline{\Gr_n}(M)$, such that if $A^m, B^n\subset M$ are $C^1$ submanifolds then
\begin{equation}\label{howard formula}
\int_G \vol_{\dim M- m-n}(A\cap gB)\, dg = \int\int_{A \times B} f_{M,G,K}(T_xA, T_yB)\, dx \, dy.
\end{equation}
\end{Theorem}
Note that the function $f_{M,G,K}$ is completely determined by its restriction $\bar f_{M,G,K}$to $\Gr_m(T_oM) \times \Gr_n(T_oM)$, where $o=[K]\in M$ is a representative point, and that this restriction is $K\times K$ invariant. Under this correspondence, the function $f_{M,G,K}$ is in a certain sense universal:

\begin{Theorem}[Transfer principle \cite{howard93}]\label{transfer principle} 
Suppose $G'$ is another unimodular Lie group containing $K$, and $M'= G'/K$ an associated Riemannian homogeneous space, such that for representative points $o\in M, o'\in M'$ there exists an isometric $K$-map $T_oM \to T_{o'}M'$. If we identify these two spaces via this map, then $\bar f_{M,G,K}= \bar f_{M',G',K}$.
\end{Theorem}
\begin{proof}[Heuristic proof] Given $A^m,B^n \subset M$, we may think of $A,B$ as being made up of infinitesimal pieces of linear elements $E \in \Gr_m(V), F\in \Gr_n(V)$, where $V:=T_oM \simeq T_xM$ for any $x\in M$. Taking Riemann sums, it follows that $\bar f_{M,G,K} = \bar f_{V, K\ltimes V,K}$.
\end{proof}

\subsection{The normal cycle, curvature measures and the first variation of a valuation}  \label{curv measures} Let $S(V)$ denote the unit sphere of $V$ and set $SV:=V \times S(V)$, the sphere bundle over $V$. Given a smooth translation-invariant form $\beta\in \Omega^{n-1}(SV)^V$ we define $\Psi_{\beta}\in\valsm(V)$ by
\begin{equation}\label{psi map}
 \Psi_{\beta}(A) :=  \int_{N(A)}{\beta}.
 \end{equation}
for $A \in \K(V)$, where $N(A)$ is the normal cycle of $A$. Conversely, any element of $\valsm(V)$ may be expressed as $c \vol_n + \Psi_{\beta}$ for some constant $c$ and some ${\beta}$ as above. This was proved by Alesker \cite{ale05a}, Thm. 5.2.1.

The map $\Psi:\Omega^{n-1}(SV)^V \to \valsm(V)$ from forms to valuations may be factored through the {\bf curvature measure map} $\Phi$ as follows. The {\bf curvature measure} $\Phi_{\beta}$ is defined to be the assignment to any $A\in \K(V)$, of a signed measure supported on  $\partial A$ given by
 $$ \Phi^A_{\beta}(S) :=\int_ {\pi^{-1}(S)\cap N(A)}  \beta$$
 for measurable subsets $S \subset V$, where $\pi:SV \to V$ is the projection.
Thus $\Psi_{\beta}(A)= \Phi^A_{\beta}(A)$. 
We say that the curvature measure  $\Phi_{\beta}$ is {\bf nonnegative}
 if the measure $\Phi_{\beta}^A \ge 0$ for all $A \in \K(V)$.
We observe that if the boundary of $A$ is a smooth hypersurface then the last integral may be expressed as the integral over $S$ of a function, determined by $\beta$, which at each point $x\in \partial A$ is polynomial in the second fundamental form of $\partial A$ at $x$ (cf. Lemma \ref{2nd ff} below).

Recall that $SV$ is a contact manifold with the global contact form $\alpha$ defined by $\alpha|_{(x,v)}(w)=\langle v,d\pi(w)\rangle$. The unique vector field $T$ on $SV$ with $i_T\alpha=1, \mathcal{L}_T\alpha=0$ is called the {\bf Reeb vector field} (here $\mathcal L$ denotes the Lie derivative). Given a form $\beta \in \Omega^{n-1}(SV)$ there exists a unique vertical form $\alpha \wedge \xi$ such that $d(\beta+\alpha \wedge \xi)$ is vertical, i.e. is a multiple of $\alpha$. The {\bf Rumin operator} $D$, introduced in \cite{rumin}, is the second order differential operator $D\beta:=d(\beta+\alpha \wedge \xi)$. 

Consider now the first variation of a valuation $\mu \in \valsm(V)$: given $A\in \K^{sm}$ and a smooth vector field $\xi$ on $V$, we put $$\delta_\xi \mu (A): = \left.\frac{d}{dt}\right|_{t=0} \mu(F_t(A))$$
where $F_t$ is the flow of $\xi$. The following implies that $\delta_\xi\mu$ extends by continuity to a smooth (but not translation-invariant) valuation in the sense of \cite{ale05a, ale05b,ale05d,ale06,alefu05} (although we will not make use of this fact).

\begin{Lemma} \label{1st variation}If $\mu = \Psi_{\beta}$ then
\begin{equation}
\delta_\xi \mu (A) = \int_{N(A)} \langle \xi,\pi_* T\rangle\, i_T(D\beta).
\end{equation}
\end{Lemma}

Since $\partial A$ is smooth this may be rephrased as

\begin{Corollary}\label{variation corollary} Suppose $A\in \K^{sm}(V)$, and let $n$ be the outward pointing normal field to $\partial A$. Then
$$
\delta_\xi\Psi_{\beta}(A) = \int_{\partial A} \langle \xi,n\rangle \, d \Phi^A_{i_T(D\beta)}
$$
\end{Corollary}

\begin{proof}[Proof of Lemma \ref{1st variation}] Let $\tilde \xi$ denote the complete lift of $\xi $ to $SV$, i.e. the vector field whose flow $\tilde F_t$ consists of contact transformations and which covers $F_t$ (\cite{yaish73}). Put $A_t:= F_t(A)$. Then $N(A_t) = \tilde F_{t*}(N(A))$,  whence
\begin{align*}
\delta_\xi \Psi_{\beta}(A) &= \left.\frac{d}{dt}\right|_{t=0} \left(\int_{N(A_t) }\beta \right)\\
&= \int_{N(A)}\mathcal L_{\tilde \xi}\beta \\
&= \int_{N(A)} \alpha(\tilde \xi)\, i_T D\beta \\
&= \int_{N(A)} \langle\xi,\pi_*T\rangle\, i_T\left(  D\beta\right),
\end{align*}
as claimed.
\end{proof}

The kernel of the map $\Psi $ of \eqref{psi map} has been characterized in \cite{bebr07}. This result may be restated in the vector space setting as follows. Define the map $\delta$ from $\valsm$ to the space of curvature measures by
\begin{equation}\label{eq: def delta}
\delta(\Psi_{\beta}):= \Phi_{i_T (D\beta)}, \quad \delta(\vol_n) := P, 
\end{equation}
where $P^K(S) = \vol_{n-1}(S\cap \partial K)$ for $K \in \Ksm$ and $S \subset V$ measurable. 
We recall that $\mu \in \Val(V)$ is said to be {\bf monotone} if $\mu(K)\le \mu(L)$ whenever $K\subset L, K,L \in \K(V)$.

\begin{Theorem} \label{1st variation map}\label{thm_monotone_general}
The mapping 
$\delta$
is well-defined, with kernel equal to the one-dimensional subspace spanned by the Euler characteristic $\chi$. A valuation $\mu\in \valsm(V)$ is monotone if and only if 
$\delta \mu  \ge 0$ and $\mu(\{point\}) \ge 0$.
\end{Theorem}
\begin{proof} That $\delta$ is well-defined follows from Lemma \ref{1st variation}. Corollary \ref{variation corollary} implies that if $\mu \in \ker \delta$ then $\delta_\xi \mu \equiv 0$ for all smooth vector fields $\xi$. Taking $\xi := -\sum x_i \frac{\partial}{\partial x_i}$ to be the Euler vector field generating the homothetic flow towards the origin, continuity implies that $\mu(K) =  \mu (\{0\})=:c$ for all $K \in \K$. It follows that $\mu = c\chi$.

To prove the last assertion, by continuity of $\mu$ it is enough to show that $\mu$ is monotone iff $\mu(\{point\})\ge 0$ and $\delta \mu^K \ge 0$ for all $K \in \K^{sm}$.

Suppose $\mu$ is monotone and $K \in K^{sm}$. Then $\mu(\{point\}) \ge 0$ since $\mu (\emptyset ) = 0$. Furthermore, if $f: \partial K \to \R$ is smooth and $\ge 0$ then by Corollary \ref{variation corollary} $0\le \delta_{fn}\mu(K) = \int_{\partial K}f\, d(\delta\mu)^K$. This implies that $(\delta \mu)^K \ge 0$, as claimed.

To prove the converse, it is enough to show that if $K,L \in \K^{sm}$ and $K\supset L$ then $\mu(K) \ge \mu(L)$. Under these conditions there is a smooth deformation $F_t:V \to V$ such that $F_0=Id$, $F_1(L) = K$ and $\langle\frac{\partial F_t}{\partial t}(t),n\rangle \ge 0$ for all outward normals $n$ to $F_t(L)$ (for example, the deformation arising from the linear interpolation between the support functions of the two bodies). Integrating the result of Corollary \ref{variation corollary} completes the proof.
\end{proof}

\subsection{Constant coefficient valuations.}\label{const coeff} 
If ${\beta}$ extends to a translation-invariant form $\overline {\beta}\in\Omega^{n-1}(TV)^V \simeq\Omega^{n-1}(V \times V)^V$, then Stokes' theorem gives
$$ \int_{N(A)} {\beta} = \int_{N_1(A)} d\overline {\beta},$$
where $N_1(A)$ is the ``disk bundle" defined in (41) of \cite{befu06}. We consider here the case where $\phi := d\overline{\beta}$ has constant coefficients, i.e. $\phi \in \Lambda^n(V\oplus V)$.

This subject is relevant here for two related reasons. First, it turns out (cf. Theorem \ref{thm_herm_int_vols} below) that all unitary-invariant valuations belong to this class. Second, constant coefficient valuations are important even in the general theory of valuations: from \eqref{eq: def delta}, we know that the first variation $\delta \mu$ of any valuation $\mu$ on $V^n$ corresponds to a translation-invariant differential form $\gamma$ of degree $n-1$ on the sphere bundle $SV$, which is a contact manifold. At each point $(x,v)\in SV$, the contact hyperplane $Q_{x,v}$ may be naturally identified with $P_v\oplus P_v$, where $P_v:= v^\perp$. Thus if we fix $(x,v)$ and restrict $\gamma_{x,v}$ to $Q_{x,v}$, we obtain an element of $\Lambda^{n-1} (P_v\oplus P_v)$. We may now regard $\gamma_{x,v}$ as giving a constant coefficient valuation on the vector space $P_v$. It turns out that the positivity of this family of ``infinitesimal"  constant coefficient valuations (parametrized by $(x,v)\in SV$) is equivalent to the monotonicity (in the sense defined in the remarks preceding Theorem  \ref{thm_monotone_general} above) of $\mu$. This has the following consequence: in view of the fact (Corollary \ref{pos_ccv}) that a constant coefficient valuation is positive iff its homogeneous components are positive, a general translation-invariant valuation is monotone iff its homogeneous components are monotone (Thm. \ref{monotone homogeneous}).

Strictly speaking, the positivity of the constant coefficient valuation determined by $\gamma_{x,v}$ is not the relevant concern for the monotonicity question--- instead, the matter turns on the positivity of the functional on symmetric bilinear forms defined in equation \eqref{lambda functional}. However, Lemma \ref{tfae +} and and Prop. \ref{nu + lambda +} show that these two conditions are equivalent. This is a help when we want to determine the monotone cone in the space of $U(n)$-invariant valuations: the family of infinitesimal constant coefficient valuations that arise in calculating their first variations may be expressed in terms of the invariant valuations in dimension $n-1$. Thus the determination of the (invariant) positive cone translates at once into a criterion (Prop. \ref{+ curvature measures}) for the monotone cone.

Put $\Sigma$ for the vector space of self-adjoint linear maps $V\to V$. We identify $\Sigma$ in the usual way with the space of symmetric bilinear forms on $V$. Given $\phi\in \Lambda^n(V \oplus V)$, consider the map $\lambda_\phi:\Sigma \to \R$ given by
\begin{equation}\label{lambda functional}
\lambda_\phi (\sigma) := \overline\sigma^*\phi,
\end{equation}
where $\overline \sigma (v ):= (v,\sigma v)$ is the graphing map, and we identify $\Lambda^n V$ with $\R$ by $t \cdot \vol\simeq t $. 
Given a euclidean space $W$ of dimension $n+1$, together with $A\in \Ksm(W)$ and $\beta \in \Omega(SW)^W$, it is convenient to express the curvature measure $\Phi_\beta^A$ in these terms by taking $V:= T_x \partial A$, where $x \in \partial A$. Let $n:\partial A \to S(W)$ denote the Gauss map and $\sigma_x:V \to V$ the Weingarten map. As above, the contact hyperplane $Q_{x,n(x)}$ is naturally identified with $V \oplus V$, and (after restriction) $\beta_{x,n(x)} \in \Lambda^n(V\oplus V)$. The following is immediate.
\begin{Lemma}\label{2nd ff} Let $\beta \in \Omega(SV)^V$. For $A\in \Ksm$ the curvature measure determined by $\beta$ may be expressed as the curvature integral
\begin{equation}
\Phi_\beta^A(S) = \int_{S\cap \partial A} \lambda_{\beta_{x,n(x)}}( \sigma_x) \, dx.
\end{equation}
$\square$
\end{Lemma}

We say that $\lambda_\phi \ge 0$ if $\lambda_\phi (\sigma) \ge 0$ whenever $\sigma $ is nonnegative semidefinite.
Put $\nu_\phi$ for the valuation
\begin{equation}\label{nu phi}
\nu_\phi(K):= \int_{N_1(K)} \phi.
\end{equation}
Put $\Lambda^n_k(V \oplus V)$ for the space of forms of bidegree $(k,n-k)$ and $\Sigma_k\subset \Sigma$ for the cone of maps of corank $k$.

Observe that if $\phi \in  \Lambda^n_k(V \oplus V)$ then the Klain function $\kl_{\nu_\phi}$ is given as follows. Given $E \in \Gr_k(V)$, let $\bar e_1,\dots,\bar e_n$ be a basis {\bf adapted to $E$}, i.e. an  orthonormal basis for $V$ such that $\bar e_1,\dots,\bar e_k$ span $E$. Put $e_i := (\bar e_i,0), \eps_i:=(0,\bar e_i)$. Then
\begin{equation}\label{klain e}
\kl_{\nu_\phi}(E) = \pm\omega_{n-k}\,\phi(e_1,\dots,e_k,\eps_{k+1},\dots, \eps_n),
\end{equation}
where the sign is positive or negative accordingly as the ordered basis $\bar e_i$ determines the correct orientation of $V$ or not.

\begin{Lemma} \label{tfae +} Suppose $\phi   \in \Lambda_k^n(V \oplus V)$. The following are equivalent:
\begin{enumerate}
\item \label{valuation +}$\nu_\phi \ge 0$.
\item \label{klain +}$\kl_{\nu_\phi} \ge 0$.
\item \label{lambda +}$\lambda_\phi \ge 0 $.
\item \label{lambda k +}$\lambda_\phi(\sigma) \ge 0$ for all $\sigma \in\Sigma_k$ with $\sigma \geq 0$.
\end{enumerate}
\end{Lemma}
\begin{proof}
\eqref{valuation +} $\iff$ \eqref{klain +}: That \eqref{valuation +} $\implies$\eqref{klain +} is obvious. To prove the converse it is enough to observe that if $P\subset V$ is a compact convex polytope then
$$ \nu_\phi(P) =\sum_{F\in P^k} \kl_{\nu_\phi}(\langle F\rangle) \vol_k(F) \angle(P,F)$$
where $P^k$ is the $k$-skeleton of $P$, $\langle F\rangle $ is the $k$-plane spanned by $F$, and $\angle (P,F)$ is the normalized exterior angle of $P$ along $F$.

\eqref{lambda k +} $\iff$ \eqref{lambda +}: That \eqref{lambda +} $\implies$ \eqref{lambda k +} is obvious. To prove the converse, let $\tau \in \Sigma, \tau \ge 0$. We may assume that $\tau $ is diagonal. The restriction of $\lambda_\phi$ to the subspace of diagonal maps $\tau$, with entries $t_1,\dots ,t_n\ge 0$, may be expressed as $\lambda_\phi(\tau) = \sum_{i_1<\dots<i_{n-k}} a_{i_1\dots i_{n-k}} t_{i_1}\dots t_{i_{n-k}} $ for some coefficients $a_{i_1\dots i_{n-k}}$. Setting suitable subsets of the $t_i$ to be zero, the hypothesis implies that all $a_{i_1\dots i_{n-k}}\ge 0$.

 \eqref{lambda k +} $\iff$ \eqref{klain +}: 
 Given $\sigma \in \Sigma_k$, $\sigma \ge 0$, let $\bar e_i$ be a positively oriented basis of $V$ adapted to $E:=\ker \sigma$. Then
 \begin{align*}\lambda_\phi(\sigma) &= \phi(e_1,\dots,e_k, e_{k+1}+a_{k+1}\eps_{k+1},\dots,e_n +a_n\eps_n)\\
&= \phi(e_1,\dots,e_k,(\bar e_{k+1},\sigma \bar e_{k+1}),\dots,(\bar e_n,\sigma \bar e_n))\\
&=\det (\sigma|_{E^\perp})\, \phi(e_1,\dots,e_k, \eps_{k+1},\dots,\eps_n).
 \end{align*}
Since $\sigma\ge 0$, the determinant is nonnegative. Thus both conditions are equivalent to the assertion that the right hand side of \eqref{klain e} is nonnegative on such a basis.
\end{proof}

\begin{Corollary} \label{pos_ccv}
A constant coefficient valuation is positive if and only if its homogeneous components are positive. 
\end{Corollary}

\proof
Let $\mu=\nu_\phi, \phi \in \Lambda^n(V \oplus V)$ be a constant coefficient valuation. Let $\phi = \sum_k \phi_k , \ \phi_k \in \Lambda_k^n(V \oplus V)$. Suppose some $\nu_{\phi_k}\not \ge 0$. By Lemma \ref{tfae +}, there is $E \in \Gr_k$ such that $\kl_{\nu_{\phi_k}}(E ) < 0$. Since the restrictions to $E$ of the $\nu_{\phi_j}, j > k,$ all vanish, it follows that $\nu_\phi(E\cap B_R) <0$ for balls of sufficiently large radius $R$. 
\endproof

\begin{Proposition}\label{nu + lambda +}
Let $\phi \in \Lambda^n(V \oplus V)$. Then $\nu_\phi \ge 0$ iff $\lambda_\phi\ge 0$.
\end{Proposition}
\begin{proof}
 This follows from Lemma \ref{tfae +}, Corollary \ref{pos_ccv} and the fact that 
 \begin{displaymath}
 \lambda_\phi\ge 0 \text{ iff each } \lambda_{\phi_k}\ge 0, 
 \end{displaymath}
whose proof is similar to that of Corollary \ref{pos_ccv}, substituting an appropriate nonnegative symmetric bilinear form of rank $k$ in place of $E$.
\end{proof}

\begin{Theorem} \label{monotone homogeneous} A valuation $\mu\in \Val(V)$ is monotone if and only if all of its homogeneous components are monotone.
\end{Theorem}
\begin{proof} First we prove the statement in the smooth case.

By \eqref{eq: def delta}, given $\mu\in \valsm$, the first variation measure of $\mu$ may be expressed as  $\delta \mu = \Phi_\gamma$ for some $\gamma \in \Omega(SV)^V$ (if $\mu$ is a multiple of $\vol$ then $\gamma $ is the corresponding multiple of the form $\kappa_{n-1}$ of \cite{fu90}).
Since the second fundamental form of a smooth convex hypersurface is nonnegative semidefinite, and conversely every nonnegative semidefinite bilinear form may be realized as such at some point of the boundary of such a hypersurface, from Lemma \ref{2nd ff} it follows that the curvature measure $\Phi_\gamma$ is nonnegative iff $\lambda_{\gamma_{(x,v)} }\ge 0$ as an element of $\Lambda^{n-1}_k(Q_{x,v})$ for every $(x,v) \in SV$. By Lemma \ref{tfae +}, this is the case iff $\lambda_{\gamma^k_{(x,v)} }\ge 0,\, k =0,\dots,n-1$, where $\gamma^k_{(x,v)}\in \Lambda^{n-1}_k(Q_{x,v})$ are the homogeneous components of $\gamma_{(x,v)}$. Thus by assertion (2) from the proof of Proposition \ref{nu + lambda +}, the present proof will be completed by showing that these correspond to the homogeneous components of $\mu$. 

This amounts to showing: if $\deg \mu = k$ then $\gamma_{x,v} \in \Lambda^{n-1}_{k-1}(Q_{x,v})$ for all $(x,v)\in SV$. Since this is clearly true when $\mu$ is a multiple of $\vol$, we may assume that $\deg \mu < n$, and hence $\mu = \Psi_\beta$ for some translation-invariant form $\beta$ of bidegree $(k,n-k-1)$. Note that $d\beta$ then has bidegree $(k,n-k)$. By the construction of \cite{rumin}, $D\beta = d(\beta + \alpha \wedge \xi)$, where $\xi$ is the unique form such that $i_T\xi = 0$ and  $\left.(d\beta + d\alpha \wedge \xi )\right|_{Q_{x,v}} = 0$. In particular $\xi $ is translation-invariant and of bidegree $(k-1,n-k-1)$, so $\gamma := i_T D\beta$ has bidegree $(k-1,n-k)$, as claimed.

Next, let $\mu$ be any continuous translation invariant valuation. Let $m_1,m_2,\ldots$ be a sequence of smooth compactly supported probability measures on ${GL}(V)$ whose supports converge to the identity. The valuations 
\begin{displaymath}
 \mu * m_i :=\int_{Gl(V)} g \mu \, dm_i(g)
\end{displaymath}
where $g\mu(A):= \mu(g^{-1}A)$,
are then smooth and monotone, with $ \mu * m_i  \to \mu$. Thus the homogeneous components of each $\mu_i$ are monotone by what we have shown above, and the resulting sequences converge, respectively, to the homogeneous components of $\mu$. Since monotonicity is clearly a closed condition, the result follows.
\end{proof}

\begin{Proposition} \label{add a line} Suppose $W$ is the orthogonal direct sum  $\R \oplus V$, with orientation induced by those of $\R, V$, and let $t,\tau:W\times W\to \R$ be the projections to the two $\R$ factors respectively. Let $\phi \in \Lambda^{n}(V\times V)$. Then the three conditions
$$ \nu_\phi \ge 0,\quad\nu_{dt\wedge\phi}\ge 0,\quad\nu_{d\tau\wedge\phi}\ge 0 $$
are equivalent. If $\phi \in \Lambda^{n}_k(V\times V)$ and $\psi \in \Lambda^{n}_{k-1}(V\times V)$ then 
$\nu_{d\tau\wedge \phi + dt\wedge \psi }\ge 0$ iff both $\nu_\phi, \nu_\psi \ge 0$.
\end{Proposition}

\begin{proof} We may assume that $\phi \in \Lambda^{n}_k(V\times V)$, so that 
$$dt\wedge \phi\in \Lambda^{n+1}_{k+1}(W\times W), \ d\tau\wedge \phi\in \Lambda^{n+1}_{k}(W\times W).$$ 
Given $E \in \Gr_{k+1}(W)$, there is a positively oriented basis of $W$ adapted to $E$ of the form
$$  c \frac{\partial}{\partial t} + s \bar e_{k+1},\bar e_1\dots,\bar e_k, -s\frac{\partial}{\partial t} + c \bar e_{k+1},\bar e_{k+2},\dots,\bar  e_n $$
where $\bar e_1,\dots, \bar e_n$ is a positively oriented orthonormal basis for $V$ and $c^2 + s^2 = 1$. Similarly, given any $F \in \Gr_{k}(W)$ there is a positively oriented basis of $W$ adapted to $F$ of the form
$$  c \frac{\partial}{\partial t} + s \bar e_1,\dots,\bar e_k, -s\frac{\partial}{\partial t} + c \bar e_1,\bar e_{k+1},\dots,(-1)^{k-1}\bar e_n .$$
By Lemma \ref{tfae +}, we may check the nonnegativity of $\nu_{dt\wedge\phi}$ by evaluating
\begin{align}
\notag dt\wedge\phi\left(c \frac{\partial}{\partial t} + s e_{k+1},e_1\dots,e_k,-s\frac{\partial}{\partial \tau} + c \eps_{k+1},\eps_{k+2},\dots,\eps_n\right)&\\
\label{dt}= dt\wedge\phi\left(c \frac{\partial}{\partial t} ,e_1,\dots,e_k, c \eps_{k+1},\eps_{k+2},\dots,\eps_n\right) &\\
\notag= c^2\phi(e_1,\dots,e_{k},\eps_{k+1},\dots,\eps_n), \, \  \quad\quad\quad\quad\quad\quad&
\end{align}
and of $\nu_{d\tau\wedge\phi}$ by evaluating
\begin{align}
\notag d\tau\wedge\phi\left(c \frac{\partial}{\partial t} + s e_1,\dots,e_k,-s\frac{\partial}{\partial \tau} + c \eps_1,\eps_{k+1},\dots,(-1)^{k-1}\eps_n\right)&\\
\label {dtau}= d\tau\wedge\phi\left(s e_1 ,\dots,e_k,-s\frac{\partial}{\partial \tau}, \eps_{k+1},\dots,(-1)^{k-1}\eps_n\right)&\\
\notag= s^2\phi(e_1,\dots,e_{k},\eps_{k+1},\eps_{k+2},\dots,\eps_n).\quad\quad\quad\quad\quad\quad \ \ \, &
\end{align}
By \eqref{klain e}, each of these expressions is nonnegative precisely when $\nu_\phi \ge 0$, which proves the first assertion.

To prove the second assertion, it is enough to show that the first condition implies the second. But  \eqref{dt} and \eqref{dtau} imply that for $E\in \Gr_k(V), \ F \in \Gr_{k-1}(V)$
\begin{align}
\omega_{n+1-k}^{-1}\kl_{\nu_{d\tau\wedge \phi + dt\wedge \psi } }(\{0\} \times E)&= \omega_{n-k}^{-1}\kl_{\nu_\phi}(E), \\
\kl_{\nu_{d\tau\wedge \phi + dt\wedge \psi } }(\R \times F )&= \kl_{\nu_\psi}(F), 
\end{align}
from which this follows at once.
\end{proof}

\section{Special bases for $\valun$}\label{bases}

Every valuation in $\valun(\C^n)$ is even and smooth.

\subsection{The monomial basis and its Fourier transform}\label{monomial basis} We recall the global valuations $s \in \valun_2, t\in \valun_1$ from \cite{fu06}. The monomials 
$$s^p t^{k-2p}, \quad 0\le p \le \min \left\{\left\lfloor \frac k 2\right\rfloor, \left\lfloor \frac {2n-k}2 \right\rfloor\right\}$$
constitute a basis of $\valun$. In Alesker's \cite{ale04} notation, 
\begin{displaymath}
s^pt^{k-2p} = \frac{(k-2p)!\omega_{k-2p}}{\pi^ {k-2p}} U_{k,p}, 
\end{displaymath}
where
\begin{equation*}
U_{k,p}(K) := \int_{\overline{\Gr}_{2n-2p,{n-p} }} \mu_{k-2p}(  K\cap \bar E) \, d\bar E
\end{equation*}
and the integral is over the corresponding affine Grassmannian with Haar measure $d\bar E$ normalized as in (19) of \cite{fu06}.
By \cite{befu06} and \cite{ale04}, their Fourier transforms are given by 
$$
\fourier {s^pt^{k-2p}} = \fourier{(s^p)}*\fourier{(t^{k-2p})} = s^{n-p}* t^{2n-k+2p}= C_{2n-k,n-p}
$$
where $*$ is the convolution product of \cite{befu06} and 
\begin{equation*}
C_{k,q}(K) := \int_{\Gr_{2q,q}} \mu_k(\pi_E(K)) \, dE.
\end{equation*}

We recall from \cite{fu06}:
\begin{Theorem}\label{relations theorem}
 The ideal of polynomials $p$ such that $p(s,t) =0$ locally at $n$ is the ideal $(f_{n+1},f_{n+2})$, where $\deg f_k(s,t) =k$ and $\log (1+s+t) = \sum_k f_k(s,t)$. The $f_k$ satisfy the relations
\begin{align}
\notag f_1 & = t\\
\notag f_2 & = s-\frac{t^2}{2}\\
\label{recursion}
ksf_k+(k+1)&tf_{k+1}+(k+2)f_{k+2}   =0, \quad k \geq 1.
\end{align} 

\end{Theorem}

\subsection{The hermitian intrinsic volumes}\label{hermitian section}
\begin{Theorem} \label{thm_herm_int_vols}
There exist global valuations $\mu_{k,q} \in \valuinf_k$ uniquely determined by the relations
\begin{equation} \label{defining_equation_mu}
\kl_{\mu_{k,q}} (E^{k',q'}) = \delta^{k',q'}_{k,q}.
\end{equation}
The valuations $\mu_{k,q}, \ \max(0,k-n)\le q\le \lfloor \frac k 2\rfloor$,  comprise a basis for the vector space $\valun_k$. 

The $\mu_{k,q}$ are all constant coefficient valuations in the sense of Section \ref{const coeff}, and satisfy the local relations
\begin{equation}\label{fmu}
\fourier{\mu_{k,q}} = \mu_{2n-k, n-k +q}.
\end{equation}
\end{Theorem}
\begin{proof}
Let $(z_1,\ldots,z_n,\zeta_1,\ldots,\zeta_n)$ be canonical coordinates
on $T\mathbb{C}^n\simeq \mathbb{C}^n \times \mathbb{C}^n$, where $z_i=x_i+\sqrt{-1}y_i$ and
$\zeta_i=\xi_i+\sqrt{-1}\eta_i$. The natural action of $U(n)$ on $T\mathbb{C}^n$ corresponds to the
diagonal action on $\mathbb{C}^n \times \mathbb{C}^n$. 

Following Park \cite{pa02}, we consider the elements 
\begin{align*}
\theta_0 & :=\sum_{i=1}^n d\xi_i \wedge d\eta_i \\
\theta_1 & := \sum_{i=1}^n \left(dx_i \wedge d\eta_i-dy_i \wedge d\xi_i\right)\\
\theta_2 & :=\sum_{i=1}^n dx_i \wedge dy_i
\end{align*}
in $\Lambda^2
(\mathbb{C}^n \oplus \mathbb{C}^n)^*$. 
Thus $\theta_2$ is the pullback via the projection map $T\mathbb{C}^n \to \mathbb{C}^n$ of the K\"ahler form of $\mathbb{C}^n$, and $\theta_0 +\theta_ 1 + \theta_2 $ is the pullback of the K\"ahler form under the exponential map $\exp(z,\zeta):= z+ \zeta$. Together with the symplectic form $\theta_s = \sum_{i =1}^n (dx_i \wedge d\xi_i + dy_i \wedge d\eta_i)$, the $\theta_i$ generate 
the algebra of all $U(n)$-invariant elements in $\Lambda^*
(\mathbb{C}^n \times \mathbb{C}^n)$ (cf. \cite{pa02}). 
 
For positive integers $k,q$ with $\max\{0,k-n\} \leq q \leq
\frac{k}{2} \leq n$, we now set  
\begin{displaymath}
\theta_{k,q}:= c_{n,k,q}\theta_0^{n+q-k}\wedge\theta_1^{k-2q}\wedge\theta_2^q.
\end{displaymath}
where
$$ c_{n,k,q}:= \frac{1}{q!(n-k+q)!(k-2q)! \omega_{2n-k}} .$$
Note that $\theta_{k,q} \in \Lambda^{2n}_k(\C^n \times \C^n)$. 
We put
\begin{equation} \label{eq_mu_integration}
\mu_{k,q}(K):=\int_{N_1(K)} \theta_{k,q}.
\end{equation}
Since this valuation has constant coefficients in the sense of section \ref{const coeff}, we may evaluate its Klain function using \eqref{klain e}.
We write $E^{k,p}$ for a generic element of $\Gr_{k,p}(\C^n)$. By invariance we may assume that $E^{k,p}= \C^p \oplus \R^{k-2p}$, with adapted basis
\begin{equation}\label{adapted basis}
\frac{\partial}{\partial z_1},\dots,\frac{\partial}{\partial z_p},\frac{\partial}{\partial x_{p+1}},\dots,\frac{\partial}{\partial x_{k-p}}, \frac{\partial}{\partial y_{p+1}},\dots,\frac{\partial}{\partial y_{k-p} },
\frac{\partial}{\partial z_{k-p+1}},\dots,\frac{\partial}{\partial z_n},
\end{equation}
where $\frac{\partial}{\partial z_i}$ stands for the pair $\frac{\partial}{\partial x_i},\frac{\partial}{\partial y_i}$. We evaluate
\begin{align}
\notag \theta_{k,q}\left(\frac{\partial}{\partial z_1},\right. &\dots, \frac{\partial}{\partial z_p},  \frac{\partial}{\partial x_{p+1}},\left.\dots,\frac{\partial}{\partial x_{k-p}}, \frac{\partial}{\partial \eta_{p+1}},\dots,\frac{\partial}{\partial \eta_{k-p} },
\frac{\partial}{\partial \zeta_{k-p+1}},\dots,\frac{\partial}{\partial \zeta_n}\right) \\
\notag &=\delta^p_q p! (n-k +p)!\theta_1^{k-2p}\left(\frac{\partial}{\partial x_{p+1}},\dots,\frac{\partial}{\partial x_{k-p}}, \frac{\partial}{\partial \eta_{p+1}},\dots,\frac{\partial}{\partial \eta_{k-p} }\right)\\
&= \frac{\pm \delta_q^p}{\omega_{2n-k}} ,
\end{align}
where the sign is that of the basis $\frac{\partial}{\partial x_{p+1}},\dots,\frac{\partial}{\partial x_{k-p}}, \frac{\partial}{\partial y_{p+1}},\dots,\frac{\partial}{\partial y_{k-p} }$ relative to the standard orientation of $\C^{k-2p}$, i.e. the same as that of the basis \eqref{adapted basis}.

This proves \eqref{defining_equation_mu}. In particular, for fixed $n$ the $\mu_{k,p}$ in the given range are linearly independent, since their Klain functions are. Since their number equals the dimension of $\valun$ they form a basis. 
 Finally, since $(E^{k,q})^\perp=E^{2n-k,n-k+q}$, the relation \eqref{fmu} is immediate, concluding the proof of Thm. \ref{thm_herm_int_vols}.
\end{proof}

As a final remark about the hermitian intrinsic volumes, we recall from Theorem \ref{relations theorem} that the kernel of the map $\valuinf_{n+1} \to \valun_{n+1}$ is spanned by the polynomial $f_{n+1}$. At the same time it is clear that $\mu_{n+1,0} =0$ locally at $n$. This implies the following global relation.
\begin{Lemma} \label{mu = cf} There are constants $\gamma_k\ne 0$ such that 
$$\mu_{k,0} = \gamma_k f_{k}.$$
\end{Lemma}

The valuation $\mu_{n,0} \in \valun $ was originally discovered by Kazarnovskii, and is called the {\bf Kazarnovskii pseudo-volume}.

\subsection{Hermitian curvature measures}
Next we consider the $\overline{U(n)}$-invariant curvature measures, which correspond to invariant $(2n-1)$-forms on the sphere bundle $S\C^n\simeq \C^n \times S^{2n-1}\subset \C^n \times \C^n\simeq T\C^n$. Consider first the following three invariant 1-forms on $T\C^n$ and their exterior derivatives:
\begin{align*}
\alpha  = \sum_{i=1}^n \xi_i dx_i+\eta_i dy_i, &\quad d\alpha = -\theta_s\\
\beta  = \sum_{i=1}^n \xi_i dy_i - \eta_i dx_i, &\quad d\beta = \theta_1\\
\gamma  = \sum_{i=1}^n \xi_i d\eta_i-\eta_i d\xi_i, &\quad d\gamma = 2\theta_0,
\end{align*}
where $\theta_s$ is the symplectic form of $\C^n\times \C^n \simeq T^*\C^n$. The restrictions of these forms to the sphere bundle $\C^n\times S^{2n-1}$, together with that of $\theta_2$,  generate the algebra of invariant forms on that space (we will not distinguish notationally the forms from their restrictions) \cite{pa02}. Thus each form of degree $2n-1$ that is a product of these forms gives rise to a $\overline{U(n)}$-invariant curvature measure. Since the contact form $\alpha$ and its exterior derivative $\theta_s$ vanish identically on any normal cycle, it is enough to consider the products of $\beta,\gamma, \theta_0,\theta_1,\theta_2$. Since $\partial N_1(K) =N(K)$, from Stokes' theorem one easily computes that

\begin{Proposition} \label{prop_hermitian_curvature_measures}
Set $B_{k,q}:=\Phi_{\beta_{k,q}}, \Gamma_{k,q}:=\Phi_{\gamma_{k,q}}$ to be the curvature measures corresponding to the invariant forms
\begin{align}
\beta_{k,q}:=c_{n,k,q}&\,\beta \wedge \theta_0^{n-k+q} \wedge \theta_1^{k-2q-1} \wedge \theta_2^q,\quad k > 2q,\\
\gamma_{k,q}:=\frac  {c_{n,k,q}} 2&\,\gamma \wedge \theta_0^{n-k+q-1} \wedge \theta_1^{k-2q} \wedge \theta_2^q, \quad n>k-q
\end{align}
on the sphere bundle $\C^n\times S^{2n-1}$. Then both of these curvature measures give rise to the hermitian intrinsic volume $\mu_{k,q}$, i.e. for $K \in \K$
$$
\mu_{k,q} (K)  = B_{k,q}^K(K)= \Gamma_{k,q}^K(K).
$$
\end{Proposition}

\subsection{Tasaki valuations} \label{tasaki section} Tasaki (\cite{tasaki00, tasaki03}) was the first to give explicit Poincar\'e-Crofton formulas for submanifolds in complex space forms. As a preparatory step, Tasaki showed that if $k \le n$ then the family of $U(n)$ orbits of $\Gr_k(\C^n)$ is in natural one-to-one correspondence with the $p$-dimensional simplex
$$0\le \theta_1\le\dots,\theta_p \le \frac \pi 2, \quad p := \left\lfloor\frac k 2\right\rfloor$$
The invariant $\Theta (E) := (\theta_1(E),\dots,\theta_p(E))$ is called the {\bf multiple K\"ahler angle} of  $E \in \Gr_k(\C^n)$, and is characterized by the condition that there is an orthonormal basis $\alpha_1,\dots,\alpha_k$ of the dual space $E^*$ such that the restriction of the K\"ahler form of $\C^n$ to $E$ is 
$$\sum_{i=1}^{\lfloor \frac k 2 \rfloor} \cos \theta_i \ \alpha_{2i-1}\wedge \alpha_{2i}.$$
Thus a subspace $E$ is of type $(k,q)$ if and only
if
\begin{displaymath}
\Theta(E)=\bigg(\underbrace{0,\ldots,0}_{q},\underbrace{\frac{\pi}{2},\ldots,\frac{\pi}{2}}_{p-q}\bigg).
\end{displaymath}
With this definition, the multiple K\"ahler angle is a global invariant in the sense of section \ref{globloc}, in that it remains the same under the natural embedding $\Gr_k(\C^n )\to \Gr_k(\C^{n+1})$. On the other hand it is easy to see that if $k>n$ then
\begin{displaymath}
\Theta(E)=\bigg(\underbrace{0,\ldots,0}_{k-n},\Theta(E^\perp)\bigg).
\end{displaymath} 
We remark that Tasaki defined the multiple K\"ahler angle to be $\Theta(E^\perp)$ in this case.

Tasaki (\cite{tasaki00}, Prop. 3) observed that if $k=2p \le n$ is even then there is an orthonormal basis $e_1,\dots,e_n$ of the hermitian vector space $\C^n$,  such that 
\begin{equation}\label{E basis} e_{1} , e_3, \dots, e_{2p-1},\cos \theta_1 \, \sqrt{-1} e_1 + \sin \theta_1 \, e_2,\dots, \cos \theta_p \, \sqrt{-1} e_{2p-1} + \sin \theta_p \, e_{2p}
\end{equation}
is an orthonormal basis for $E$ as a real euclidean vector space, and
\begin{align*}\label{E basis} \sqrt{-1} \,&e_{2},\dots \sqrt{-1} \,e_{2p} , e_{2p+1}, \sqrt{-1} e_{2p+1},\dots, e_{n}, \sqrt{-1} e_{n},\\
&-\sin \theta_1 \, \sqrt{-1} e_1 + \cos \theta_1 \,e_2,\dots,  -\sin \theta_p \, \sqrt{-1} e_{2p-1} + \cos \theta_p \, e_{2p}
\end{align*}
is an orthonormal basis for $E^\perp$, and similarly if $k$ is odd and/or larger than $n$. By \eqref{klain e}, it is now easy to see 
\begin{Lemma}\label{elem sym basis} For each $k,q$ as above, the Klain function 
$\kl_{\mu_{k,q}}(E)$ is a linear combination of the elementary symmetric functions in $\cos^2 \theta_1(E),\dots, \cos^2\theta_p(E)$.
\end{Lemma}
\begin{proof} Referring to the basis \eqref{E basis} and the expression \eqref{klain e} for the Klain function, the latter is symmetric in these quantities, and of degree at most one in each of them.
\end{proof}

We now define the {\bf Tasaki valuations} $\tau_{k,q}\in \valuinf, 0\le q \le p:=\lfloor \frac k 2\rfloor$ by the condition
\begin{equation}
\kl_{\tau_{k,q}}(E) = \sigma_q(\Theta(E)):=\sigma_q(\cos^2\theta_1(E),\dots,\cos^2\theta_p(E))
\end{equation}
where $\sigma_q$ is the the $q$th elementary symmetric function.

\begin{Definition} $$ u:= 4s-t^2.$$
\end{Definition}

\begin{Proposition}\label{tasaki vals}The Tasaki valuations are well-defined, and are given by
\begin{align} \label{eq_def_sigma}
\tau_{k,q}&=\sum_{i=q}^{\lfloor k/2\rfloor} \binom{i}{q} \mu_{k,i}  \\
\label{eq_mu_tu}&=
\frac{\pi^k}{\omega_k(k-2q)! (2q)!} t^{k-2q}u^q.
\end{align}
Furthermore the polynomials from Theorem \ref{relations theorem} may be expressed
\begin{equation} \label{eq_fu_polynomial_tu}
f_k =\frac 1{ k(-2)^{k-1} }  \sum_{q=0}^{\lfloor k/2\rfloor} (-1)^q\binom{k}{2q} t^{k-2q}u^q.
\end{equation}
\end{Proposition}
\begin{proof} Since the elementary symmetric functions corresponding to the $p+1$ hermitian intrinsic volumes are linearily independent, the relation \eqref{eq_def_sigma} is a straightforward calculation, using the defining relations \eqref{defining_equation_mu}.

To prove \eqref{eq_fu_polynomial_tu}, we introduce the formal complex variable $z: = t + \sqrt{-1}v$, where $v$ is formally real and $v^2 =u$. Then
\begin{align*}
\sum_k f_k = \log( 1 + s + t) &= \log \left(1 + t + \frac{t^2} 4 +\frac {v^2} 4\right)\\
& = \log\left(\left| 1 + \frac z 2 \right|^2 \right)\\
& =2\operatorname{Re}\left( \log\left( 1 + \frac z 2 \right) \right)\\
& =  \operatorname{Re}\sum_k \frac 1{ k(-2)^{k-1} } \left(t + \sqrt{-u}\right)^k .
\end{align*}

We postpone the proof of \eqref{eq_mu_tu} to section \ref{section_eigenvalues}.
\end{proof}

\begin{Corollary} \label{mu tasaki}The global Tasaki valuations $\tau_{k,q}, 0\le q\le \frac k 2$, constitute a basis for $\valuinf_k$, and in fact
\begin{equation} \label{eq_mu_in_terms_of_sigma}
\mu_{k,q}=\sum_{i=q}^{\lfloor k/2\rfloor} (-1)^{i+q} \binom{i}{q} \tau_{k,i}.
\end{equation}
If $k\le n$ then the local Tasaki valuations $\fourier{\tau_{k,q}}, 0\le q\le \frac k 2$, constitute a basis for $\valun_{2n-k}$.
\end{Corollary}

If we now write out the $U(n)$ kinematic formula in terms of the basis $\{\tau_{k,q},\fourier{\tau_{k,q}}: k\le n,0\le q\le \frac k 2\}$ of $\valun$, the general Crofton formula \eqref {poincare formula} now yields the main theorem of \cite{tasaki03}.

\begin{Theorem}[Tasaki]\label{tasaki poly} For $0\le k\le n$, there is a $\left(\lfloor \frac k 2\rfloor +1\right)\times \left(\lfloor \frac k 2\rfloor +1\right)$ symmetric matrix $T^n_k$, such that whenever $M,N\subset \C^n$ are $C^1$ compact submanifolds of dimensions $k,2n-k$ respectively,
\begin{equation}\label{tasaki formula}\int_{\overline{U(n)}} \# (M \cap \bar g N )\, d\bar g = \sum_{i,j} (T^n_k)_{i,j}\int_M \sigma_i(\Theta(T_xM))\, dx \int_N \sigma_j( \left[\Theta(T_yN)\right]^\perp)\, dy . 
\end{equation}
\end{Theorem}

The symmetry of $T^n_k$ follows from that of the pairing \eqref{fourier pairing}.
In fact these formulas exhibit a further remarkable symmetry:

\begin{Theorem}\label{palindrome thm}
If $k=2l$ is even then 
\begin{equation} \label{eq_palindromic}
(T^n_k)_{i,j}=(T^n_k)_{l-i,l-j}, \quad 0\le i,j\le l.
\end{equation}
\end{Theorem}

To prove Theorem \ref{palindrome thm} we introduce the linear involution $\iota:\valuinf_{2*} \to \valuinf_{2*}$ on the subspace of valuations of even degree, determined by its action on Tasaki valuations:
$$
\iota(\tau_{2l,q}):=\tau_{2l,l-q}.
$$

\begin{Lemma}\label{algebra aut} 
\begin{enumerate}
\item \label{multiplicative}$\iota$ is an algebra automorphism.
\item \label{iota is local} $\iota$ covers an algebra automorphism of every $\valun_{2*}$. 
\item \label{top is trivial}  The action of $\iota$ on the top degree component $\valun_{2n}$ is trivial.
\item
\label{commute}
$\iota$ commutes with the Fourier transform.
\end{enumerate}
\end{Lemma}
\begin{proof}[Proof of Lemma \ref{algebra aut}]
(\ref{multiplicative}): Any element of $\valuinf_{2*}$ may be expressed as polynomial in $t$ and $v$, involving only even powers of each variable. We may regard this as a (real) polynomial function $p(z)$ in the complex variable $z= t + \sqrt{-1}v$. From the expression \eqref{eq_mu_tu}, in these terms $\iota(p(z)) = p(\sqrt{-1}\,  z)$, which is of course an algebra isomorphism.

(\ref{iota is local}): To prove that $\iota$ descends to an automorphism of $\valun$ it is enough to show that $\iota$ stabilizes the kernel of the map $\valuinf_{2*}\to \valun_{2*}$. This kernel consists of the even degree elements of the ideal $(f_{n+1},f_{n+2})$. By (\ref{multiplicative}), it is enough to show that $\iota(f_{2k}) \in (f_{2k})$ and that $\iota(tf_{2k-1}) \in (tf_{2k-1},f_{2k})$. 
 But by the proof of \eqref{eq_fu_polynomial_tu}, 
\begin{align*}
\iota(f_{2k}) &=-\frac 1{k\,2^{2k}} \, \iota(\operatorname{Re} z^{2k}) \\
&= -\frac 1{k\,2^{2k}} \,\operatorname{Re} (\sqrt{-1} z)^{2k} \\
&= \frac {(-1)^{k+1}}{k\,2^{2k}} \,\operatorname{Re} z^{2k} \\
&=(-1)^k f_{2k}
\end{align*}
and
\begin{align*}
\iota(tf_{2k-1}) &=\frac 1{(2k-1)\,2^{2k-2}} \, \iota(\operatorname{Re}[ t z^{2k-1}]) \\
&= \frac 1{(2k-1)\,2^{2k-2}}\,\operatorname{Re}[- v(\sqrt{-1}\,z)^{2k-1}] \\
&= \frac 1{(2k-1)\,2^{2k-2}}\,\operatorname{Re}\left[( \sqrt{-1}\,z)^{2k} -\sqrt{-1}\,t(\sqrt{-1}\,z)^{2k-1}\right] \\
&=(-1)^{k+1}\frac{4k}{2k-1} f_{2k} +  \frac {(-1)^{k+1}}{(2k-1)2^{2k-2}} \,\operatorname{Re}(t
z^{2k-1})\\
&=(-1)^{k+1}\left(\frac{4k}{2k-1} f_{2k} +  t f_{2k-1}\right).
\end{align*}

(\ref{top is trivial}):  Since locally $\mu_{2n,k} =0$ for $k<n$, \eqref{eq_def_sigma} shows that $\tau_{2n,k}=\tau_{2n,n-k}=\binom{n}{k}\mu_{2n,n}$ locally. 

(\ref{commute}):
Put $\Sigma_p$ for the vector space spanned by the elementary symmetric polynomials $\sigma_{p,0}:=1,\sigma_{p,1}:= x_1 +\dots+ x_p,\dots,\sigma_{p,p} :=  x_1x_2\dots x_p$ in the $p$ variables $x_1,\dots,x_p$. As noted above, $\Sigma_p$ is canonically isomorphic to $\valuinf_{2p}$ via $\sigma_{p,q} \mapsto \tau_{2p,q}$, where the map $\iota$ corresponds to $\sigma_{p,q}\mapsto \sigma_{p,p-q}$, which we again denote by $\iota$.

Fixing $n \ge 2p$, the Fourier transform  $\ \widehat{}:\valun_{2n-2p} \to \valun_{2p}$ corresponds to the linear surjection $r: \Sigma_{n-p}\to \Sigma_p$ given by
 $$
 r (f) = f(x_1,\dots, x_p, 1,\dots,1).
 $$
 The assertion thus reduces to the claim that for $m=n-p\ge p$ the diagram
 $$ \begin{CD}
\Sigma_{m} @>{\iota}>> \Sigma_{m}\\
@VV{r}V    @VV{r}V\\
\Sigma_p @>{\iota}>> \Sigma_p
\end{CD}$$
commutes.
 It is enough to prove this for $m= p+1$, in which case $r(\sigma_{p+1,i}) = \sigma_{p,i} + \sigma_{p,i-1}$. Hence for $ i=0,\dots, p+1$,
 $$
 \iota \circ r (\sigma_{p+1,i})  = \iota(\sigma_{p,i} + \sigma_{p,i-1}) = \sigma_{p,{p-i}} +\sigma_{p,p-i+1} = r(\sigma_{p+1,p-i+1} )= r\circ \iota(\sigma_{p+1,i} ).
 $$
 \end{proof}

\begin{proof}[Proof of Thm. \ref{palindrome thm}] 
By Lemma \ref{algebra aut},
\begin{align*}
\tau_{2p,i}\cdot \fourier{\tau_{2p,j}}&= \tau_{2p,i}\cdot \fourier{(\iota\tau_{2p,p-j})} \\
&= \tau_{2p,i}\cdot {\iota(\fourier{\tau_{2p,p-j}})}\\
&=\iota\left( \tau_{2p,i}\cdot {\iota(\fourier{\tau_{2p,p-j}})}\right)\\
&= \iota(\tau_{2p,i})\cdot \fourier{\tau_{2p,p-j}}\\
&= \tau_{2p,p-i}\cdot \fourier{\tau_{2p,p-j}}.
\end{align*}
With Theorem \ref{g-kinematic}, this implies the result.
\end{proof}

\section{The positive, monotone and Crofton-positive cones} We wish to determine the cones $CP \subset M \subset P \subset \valun$ given by
\begin{align}
P&:= \{ \phi: \phi(K) \ge 0 \text{ for all } K \in \K\}, \\
M&:= \{\phi: \phi(K) \ge \phi(L) \text{ whenever } K, L \in \K \text{ and } K \supset L\},\\
CP&:= \{\phi: \text{ the homogeneous components of }\phi \text{ each admit }\\
&\hskip 1in \text{a nonnegative Crofton measure}\}.
\end{align}

We recall from \cite{befu06} that if $\phi, \psi \in \valsmplus_k$ are even, and $m_\psi $ is a Crofton measure for $\psi$, then the pairing \eqref{fourier pairing} is given by
\begin{equation}\label{integral pairing}
\langle\phi,\psi\rangle = \int_{\Gr_k} \kl_\phi \, dm_\psi.
\end{equation}

\begin{Proposition}\label{p cone} The cone $P$ is generated by the hermitian intrinsic volumes $\mu_{k,q}$. The cone $CP$ is the cone $P^*:= \{\phi: \langle \phi, \mu\rangle \ge 0 \text{ for all } \mu \in P\}$ dual to $P$ with respect to the pairing $\langle \cdot,\cdot\rangle$ of \eqref{fourier pairing}.
\end{Proposition}
\begin{proof}

By the equivalence (1)$\iff$ (3) $\iff$ (4) in Lemma \ref{tfae +}, a constant coefficient valuation belongs to $P$ iff its homogeneous components do; and by the equivalence (1) $\iff$ (2), a homogeneous constant coefficient valuation belongs to $P$ iff its Klain function is nonnegative. By Lemma \ref{elem sym basis},
the first assertion of Prop. \ref{p cone} is equivalent to the following statement. Consider the vector space $\Sigma$ spanned by the elementary symmetric functions in the variables $x_1,\dots,x_p$, and let $C$ denote the cube $0\le x_1,\dots, x_p\le 1$ (we think of $x_i= \cos^2\theta_i$). Let $f \in \Sigma$ be given. Then $\left.f\right|_C \ge 0$ iff its value at each vertex of $C\ge 0$. This is easily proved by induction on the dimension of the faces of $C$, using the observation that $f$ is affine in each variable separately if the others are held fixed.

Moving on to $CP$, put $\nu_{k,p}\in \valun$ for the dual basis to $\mu_{k,p}$ with respect to the pairing \eqref{fourier pairing}, i.e.
$$
\langle \nu_{k,p}, \mu_{l,q}\rangle:= \delta^{k,p}_{l,q}.
$$
Thus, by \eqref{integral pairing}, 
 $\nu_{k,p}$ is the valuation determined by the  Crofton measure that is $U(n)$-invariant, is supported on $\Gr_{k,p}$, and has total mass 1; furthermore it is clear that the dual cone $P^*$ is generated by the $\nu_{k,p}\in CP$, so $P^* \subset CP$.
 To prove the opposite inclusion, we note that  \eqref{integral pairing} implies that if $\psi \in CP$ then $\langle\phi,\psi\rangle \ge 0$ for all $\phi $ with nonnegative Klain function. 
Taking $\psi$ to have degree $k$ and writing $\psi = \sum_p b_p\nu_{k,p}$, we find that  
\begin{equation}\label{linear}
0\le \left \langle\sum_p a_p \mu_{k,p}, \psi\right\rangle = \sum_p a_p b_p,
\end{equation}
whenever all $a_p \ge 0$, which implies that all $b_p \ge 0$, i.e. $\psi\in P^*$.
\end{proof}

This discussion invites the following brief excursion.
Define the norms $\norm{\cdot}_\infty$ and $\norm{\cdot}_1$ on $\valsmplus$ by
\begin{align}\norm{\phi}_\infty &:= \norm{ \kl_\phi}_\infty, \\
\norm{\phi}_1&:= \min\{\operatorname{mass} m: m\text{ is a Crofton measure for }\phi\}.
\end{align}
By \eqref{integral pairing},
 the norm dual to $\norm{\cdot}_\infty$ with respect to the pairing $\langle \cdot,\cdot\rangle$ satisfies 
\begin{equation}\label{norm* <= norm}
\norm{\cdot}_\infty^* \le \norm{\cdot}_1.
\end{equation}

\begin{Proposition}
Restricted to $\valun_k$ the norms $\norm{\cdot}_1$ and $\norm{\cdot}_\infty$ are dual to one another with respect to the pairing \eqref{fourier pairing}, with
\begin{align}\label{norm infty}
\norm{\sum_p a_p\mu_{k,p}}_\infty &= \max_p |a_p|,\\
\label{norm one}\norm{\sum_p b_p\nu_{k,p}}_1 &= \sum_p |b_p|. 
\end{align}
\end{Proposition}

\begin{proof}  The relation \eqref{norm infty} follows from the argument in the first paragraph of the Proof of Prop. \ref{p cone}, and by \eqref{linear},
$$
\norm{\sum_p b_p\nu_{k,p}}_\infty^* = \sum_p |b_p| = \operatorname{mass} \left(\sum_p b_p\nu_{k,p}\right)\ge
\norm{\sum_p b_p\nu_{k,p}}_1
$$
which, with \eqref{norm* <= norm}, completes the proof.
\end{proof}

\subsubsection{The monotone cone} 

\begin{Theorem} \label{thm_monotone}
A valuation $\mu \in \valun_k$ is monotone iff
\begin{displaymath}
\mu=\sum_{q=\max\{0,k-n\}}^{\lfloor k/2\rfloor} a_{q}\mu_{k,q},
\end{displaymath}
where
\begin{equation} \label{eq_ineq_first_type}
(k-2q)a_{q} \geq (k-2q-1)a_{q+1}, \quad \max\{0,k-n\} \leq q \leq \left\lfloor \frac{k-1}{2}\right\rfloor
\end{equation}
\begin{equation} \label{eq_ineq_second_type}
(n+q-k+1) a_{q} \leq (n+q-k+3/2) a_{q+1}, \quad \max\{0,k-n-1\} \leq q \leq \left\lfloor \frac{k-2}{2}\right\rfloor.
\end{equation}
\end{Theorem}

\begin{Corollary} 
The inclusions $CP \subset M \subset P$ are strict. 
\end{Corollary}

By Theorem \ref{thm_monotone_general}, in order to prove Theorem \ref{thm_monotone} we need to characterize the cone of nonnegative hermitian curvature measures.

\begin{Proposition}\label{+ curvature measures} Given constants $a_{k,q},b_{l,p} \in \R$, $k>2q,n>l-p$ , the curvature measure ${\sum a_{k,q}B_{k,q} + \sum b_{l,p}\Gamma_{l,p}}\ge 0$ iff all $a_{k,q},b_{l,p}\ge 0$.
\end{Proposition}

\begin{proof} 
Each tangent space $T_x\partial A$ is naturally isomorphic to the orthogonal direct sum $\R \oplus \C^{n-1}$, where the first summand corresponds to the distinguished line spanned by$\sqrt{-1}\, n(x)$ and the second summand to the maximal complex subspace of $T_x\partial A$. Thus the 1-forms $\beta, \gamma$ correspond respectively to $dt, d\tau$ in Proposition \ref{add a line}. In view of the characterization in Proposition \ref{p cone} of the nonnegative elements of $\Val^{U(n-1)}(\C^{n-1})$, the result now follows from Propositions \ref{nu + lambda +} and \ref{add a line}.
\end{proof}

Recall from Theorem \ref{1st variation map} the first variation map $\delta$ from valuations to curvature measures.

\begin{Proposition}
\begin{align}
\delta \mu_{k,q} &= 2c_{n,k,q}(c_{n,k-1,q}^{-1} (k-2q)^2 \Gamma_{k-1,q} -  
c_{n,k-1,q-1}^{-1} (n+q-k)q \Gamma_{k-1,q-1} \\
+ & c_{n,k-1 ,q-1}^{-1}(n+q-k+\frac 1 2)q B_{k-1,q-1}-
 c_{n,k-1,q}^{-1} (k-2q)(k-2q-1) B_{k-1,q}
)
\end{align}
\end{Proposition}

\proof
By definition of the $\mu_{k,q}$, this valuation is represented by some $(2n-1)$-form $\omega_{k,q}$ with 
\begin{equation}\label{little d omega}
d\omega_{k,q}=c_{n,k,q} \theta_0^{n+q-k}\wedge\theta_1^{k-2q}\wedge\theta_2^q,
\end{equation}
i.e. $\mu_{k,q} =\Psi_{\omega_{k,q}}$.
To compute $D\omega_{k,q}$, we must solve for $\xi $ in the equation 
\begin{equation} \label{eq_finding_xi}
D\omega_{k,q}=d(\omega_{k,q}+\alpha \wedge \xi) \equiv 0 \mod \alpha. 	
\end{equation}
Fixing a point $(x,v)\in S\C^n=\C^n\times S^{2n-1}$, let $Q\subset T_{(x,v)} S\C^n$ denote the contact hyperplane $\alpha_{(x,v)}^\perp$. Thus $Q\simeq \R \oplus \C^{n-1}\oplus \R \oplus \C^{n-1}$ in a natural way, and carries a symplectic structure (cf. \cite{rumin}). Let $L$ denote the Lefschetz operator on $\Lambda^*Q$ (i.e. multiplication by the symplectic form $\theta_s=-d\alpha$) and $\Lambda$ the dual Lefschetz operator. By \cite{huyb04}, they induce an $\mathfrak{sl}_2$-structure on $\Lambda^*Q$, i.e. $[L,\Lambda]=k+1-2n$ on $\Lambda^kQ$. 

To solve \eqref{eq_finding_xi} amounts to finding $\xi \in \Lambda^{2n-2}Q$ with 
\begin{displaymath}
L\xi=d\omega_{k,q}|_Q	.
\end{displaymath}
We write $d\omega_{k,q}|_Q$ in terms of its Lefschetz decomposition
\begin{equation} \label{eq_lefschetz_decomposition}
d\omega_{k,q}|_Q=\sum_{i=0}^{n-1} L^{n-i}\pi_{2i}. 	
\end{equation}
Here $\pi_{2i}$ is a primitive form of degree $2i$, i.e. $\Lambda \pi_{2i}=0$, where $\Lambda$ is the dual Lefschetz operator. The sum terminates with $i=n-1$ (and not with $i=n$) since there are no primitive forms of degree $2n$. 
Clearly
\begin{displaymath}
\xi=\sum_{i=0}^{n-1} L^{n-i-1}\pi_{2i}  	
\end{displaymath}
solves \eqref{eq_finding_xi}. 

We apply $\Lambda$ to both sides of \eqref{eq_lefschetz_decomposition} and use the fact that 
\begin{displaymath}
[L^i,\Lambda]=i(k+i-2n) L^{i-1} \quad \text{on } \Lambda^kQ
\end{displaymath}
to deduce that 
\begin{equation}\label{xi mod da}
\Lambda d\omega_{k,q}|_Q=\sum_{i=0}^{n-1} (n-i)^2 L^{n-i-1}\pi_{2i} \equiv \xi	\mod d\alpha.
\end{equation}

From this point on we drop the $\wedge$ notation, with all products of forms understood to be wedge products.

\begin{Lemma} 
\begin{equation*}
\Lambda \left(\theta_0^a \theta_1^b \theta_2^c\right) \equiv \beta \gamma\theta_0^{a-1}\theta_1^{b-2} \theta_2^{c-1}\left(ac  \theta_1^2-b(b-1)  \theta_0 \theta_2\right) \mod (\alpha, d\alpha ).
\end{equation*}
\end{Lemma}
\begin{proof} Since everything is $U(n)$-invariant, it suffices to do the computation at the point $(0,e_1)\in S\C^n$, i.e. where $\xi_1=1,\xi_2=\ldots=\eta_n=0$. At this point, $d \xi_1=0$, $\frac{\partial}{\partial \xi_1}=0$ since $\sum (\xi_j^2+\eta_j^2)=1$, and $\beta=dy_1, \gamma=d\eta_1$.  

Next, using the abbreviation $i_{x_j}:=i_{\frac{\partial}{\partial x_j}}$, we compute that 
\begin{align*}
i_{\xi_j} \circ i_{x_j} \left(\theta_0^a \theta_1^b \theta_2^c \right)& = i_{\xi_j}\left(b \theta_0^a d\eta_j \theta_1^{b-1}\theta_2^c + c \theta_0^a\theta_1^b dy_j \theta_2^{c-1}\right)\\
& = dy_j d\eta_j\left(b(b-1)  \theta_0^a \theta_1^{b-2} \theta_2^c - ac  \theta_0^{a-1}\theta_1^b \theta_2^{c-1} \right)
\end{align*}
and similarly
\begin{equation*}
i_{\eta_j} \circ i_{y_j} \left(\theta_0^a \theta_1^b \theta_2^c \right) = dx_j d\xi_j\left(b(b-1)  \theta_0^a \theta_1^{b-2} \theta_2^c-ac \theta_1^{a-1}\theta_1^b \theta_2^{c-1} \right).
\end{equation*}
Since $\Lambda=i_{\eta_1} \circ i_{y_1}+\sum_{j=2}^n (i_{\xi_j} \circ i_{x_j} +i_{\eta_j} \circ i_{y_j})$ at the selected point, and 
$\beta\gamma + \sum_{j=2}( dx_jd\xi_j +dy_jd\eta_j) = - d\alpha$, the result follows.
\end{proof}

With \eqref {xi mod da} and the defining relation \eqref{little d omega}, this yields
\begin{align*}
\xi  \equiv c_{n,k,q} \beta\gamma \theta_0^{n+q-k-1} \theta_1^{k-2q-2} \theta_2^{q-1}\left((n+q-k)q  \theta_1^{2}\right.
&\left.-(k-2q)(k-2q-1)  \theta_0\theta_2 \right)\\
\mod (\alpha,d\alpha). 	
\end{align*} 
Replacing this into \eqref{eq_finding_xi} we find 
\begin{align}
 i_T D\omega_{k,q} & \equiv i_T d\omega_{k,q}- d\xi \nonumber \\
& \equiv  c_{n,k,q}\theta_0^{n+q-k-1}\theta_1^{k-2q-2}\theta_2^{q-1} \\
& \quad\left((k-2q)^2 \gamma  \theta_0   \theta_1   \theta_2
   -(n+q-k)q \,\gamma    \theta_{1}^3\right.  \nonumber \\ 
& \quad +2(n+q-k+1/2)q\, \beta  \theta_{0}  \theta_{1}^2   \nonumber \\
&\left. \quad - 2(k-2q)(k-2q-1)  \beta  \theta_{0}^2   \theta_{2}\right)  \mod (\alpha,d\alpha) \label{eq_rumin_mod_dalpha}
\end{align}
The proposition now follows from Theorem \ref{1st variation map} and the definition of $B,\Gamma$ from Proposition \ref{prop_hermitian_curvature_measures}.  
\endproof

\proof[Proof of Theorem \ref{thm_monotone}]
Let $\mu=\sum_{k,q} a_q \mu_{k,q}$. The coefficient of $\Gamma_{k-1,q}$ with $\max\{0,k-n\} \leq q \leq \left\lfloor \frac{k-1}{2}\right\rfloor$ in $\delta \mu$ is given by 
\begin{displaymath}
2\frac{c_{n,k,q}(k-2q)^2}{c_{n,k-1,q}}a_q-2\frac{c_{n,k,q+1}(n+q-k+1)(q+1)}{c_{n,k-1,q}}a_{q+1};
\end{displaymath}
it has the same sign as $(k-2q)a_q-(k-2q-1)a_{q+1}$. 

Similarly, the coefficient of $B_{k-1,q}$ with $\max\{0,k-n-1\} \leq q \leq \left\lfloor \frac{k-2}{2}\right\rfloor$ in $\delta \mu$ is given by 
\begin{displaymath}
\frac{2c_{n,k,q+1}(n+q-k+3/2)(q+1)}{c_{n,k-1,q}}a_{q+1}-\frac{2c_{n,k,q}(k-2q)(k-2q-1)}{c_{n,k-1,q}} a_q;
\end{displaymath}
which has the same sign as $(n+q-k+3/2)a_{q+1}-(n-k+q+1)a_q$. 
By Theorem \ref{thm_monotone_general} and Proposition \ref{+ curvature measures}, the valuation $\mu$ is monotone if and only if the inequalities \eqref{eq_ineq_first_type} and \eqref{eq_ineq_second_type} are satisfied. 
\endproof


\section{Explicit kinematic formulas}

Our goal in this section is to give explicit forms for the kinematic formulas \eqref{general kf} in terms of the basis of Tasaki valuations and their Fourier transforms.
Our approach is based on the explicit calculation of the structure of $\valun$ as an $\sltwo$ module. The existence of such a structure follows from general considerations (the Jacobson-Morozov theorem \cite{del80}) and the fact, originally established by Alesker \cite{ale03b}, \cite{ale04a}, that $\valsm$ satisfies the hard Lefschetz property with respect to either of two different operators of degrees $\pm 1$ respectively. Using the results of \cite{bebr07,befu06} we compute explicitly how these operators act on the Tasaki valuations, and show that together they yield a representation of $\sltwo$ on $\valun$ (although Alesker has pointed out that this is not the case when these operators are regarded as acting on the entire space $\valsm$). We then calculate explicitly the primitive elements of $\valun$ with respect to this representation, giving rise to one more canonical basis $\pi_{k,p}$ for $\valun$. Since the Poincar\'e duality multiplication table of $\valun$ in terms of this basis is antidiagonal (Prop. \ref{diagonalization} below), we can then easily express the kinematic formulas in these terms.

\subsection{The $\sltwo$ action}
We recall \cite{ale03b,bebr07,befu06} the two operators $\tilde L,\tilde \Lambda:\valsm(\C^n)\to \valsm(\C^n)$, of degrees $\pm1$ respectively:
\begin{align}
\tilde L\phi&:= \mu_1 \cdot \phi ,\\
\tilde \Lambda \phi&:= 2 \,\mu_{2n-1}* \phi = \left.\frac{d}{dt}\right|_{t=0} \phi (\cdot + t B),
\end{align}
where $B$ is the unit ball of $\C^n$. (Note that $\tilde \Lambda \phi$ is the valuation corresponding to the curvature measure $\delta \phi$, i.e. $\tilde \Lambda \phi (A)=(\delta \phi)^A(A)$.)  

Since $\tilde L$ is a multiplication operator in a commutative algebra, the following point is obvious:
\begin{Lemma} \label{obvious lemma}For $\phi,\psi \in \valun$,
$$(\tilde L\phi)\cdot \psi = \phi \cdot(\tilde L \psi). $$
\end{Lemma}

We renormalize these operators by taking
\begin{align}
L&:= \frac{2\omega_k}{\omega_{k+1}} \tilde L ,\\
\label{tilde lambda}\Lambda &:= \frac{\omega_{2n-k}}{\omega_{2n-k+1}} \tilde \Lambda
\end{align}
on each homogeneous component $\valsm_k$. 

\begin{Lemma} \label{operators on tau}
\begin{align}
\label{L tau} L \tau_{k,p} &= (k-2p+1) \,\tau_{k+1,p},\\
\label {lambda tau}\Lambda \tau_{k,p} &= (2n-2p-k+1)\, \tau_{k-1,p} + (k-2p+1) \,\tau_{k-1,p-1}.
\end{align} 
\end{Lemma}
\begin{proof}
We show first that
\begin{align} 
\label{lambda} \Lambda \mu_{k,q} & =2(n-k+q+1) \mu_{k-1,q} + (k-2q+1) \mu_{k-1,q-1},\\
\label{ell_on_mu} L \mu_{k,q} & = 2(q+1)\mu_{k+1,q+1}+(k-2q+1) \mu_{k+1,q}. 
\end{align}

Recall from \cite{bebr07} that if $\mu(K), \mu \in \valsm(V)$, is obtained by integration over $N_1(K)$ of a differential form $\psi$ on $TV$ then $\tilde \Lambda \mu(K)$ is obtained by integration of the Lie derivative $\mathcal L_T\psi$ with respect to the Reeb vector field $T$; i.e. in the notation of \eqref{nu phi},
$$
\tilde \Lambda \nu_\psi = \nu_{\mathcal L_T\psi}.
$$
 The Lie derivatives of the $\theta_i$ with respect to $T$ are
\begin{displaymath}
\mathcal{L}_T \theta_0=0, \mathcal{L}_T \theta_1=2\theta_0,
\mathcal{L}_T \theta_2=\theta_1,
\end{displaymath}
from which one computes that 
\begin{displaymath}
\mathcal{L}_T \theta_{k,q}=\frac{\omega_{2n-k+1}}{\omega_{2n-k}} \left( 2(n-k+q+1) \theta_{k-1,q} + (k-2q+1) \theta_{k-1,q-1} \right).
\end{displaymath} 

The relation \eqref{lambda} now follows at once. 
Relation \eqref{ell_on_mu} follows from \eqref{lambda} using Equation \eqref{fmu} and the fact (which follows at once from Corollary 1.9 of \cite{befu06}) that the Fourier transform intertwines the operators $L,\Lambda$:
\begin{equation}\label{intertwine}
\widehat {}\circ L = \Lambda\circ \widehat {}.
\end{equation}
The assertions of the lemma now follow from \eqref{eq_def_sigma} and \eqref{eq_mu_in_terms_of_sigma}.
\end{proof}

\begin{Theorem} \label{thm_sl2_representation}
Let $X,Y,H$ with $[X,Y]=H, [H,X]=2X, [H,Y]=-2Y$ be generators of
  $\mathfrak{sl}(2,\mathbb{R})$. The map 
\begin{align*}
H & \mapsto 2k-2n\\
X & \mapsto L\\
Y & \mapsto \Lambda 
\end{align*}
defines a representation of $\mathfrak{sl}(2,\mathbb{R})$ on $\valun$. 
\end{Theorem}
\begin{proof} This is a direct calculation, using Lemma \ref{operators on tau}.
\end{proof}

The following corollary is a standard fact for $\sltwo$ representations, compare \cite{huyb04} or \cite{grha78}. 

\begin{Corollary}
\begin{align}
[H,L^i] & =2i L^i\\
[L^i , \Lambda] & = i L^{i-1} \circ H + i(i-1) L^{i-1}. \label{iterated_commutators}
\end{align}
\end{Corollary}

We recall that an element $\pi$ in degree $k \le n$ of such a representation is called {\it primitive} if $\Lambda \pi = 0$, or equivalently, if $L^{2n-2k+1}\pi =0$.
By the Hard Lefschetz Theorem of Alesker \cite{ale03b}, and comparing dimensions, it follows that there exists a unique (up to a multiplicative constant) primitive valuation in $\valun$ in each even degree not larger than $n$. 

In the following, we use the standard notation $(2k+1)!!=(2k+1)\cdot (2k-1) \cdot (2k-3) \cdots 1$ and set formally $(-1)!!:=1$.  
 For $0\le 2r \le n$, using Lemma \ref{operators on tau} we put 

\begin{equation}
\label{pi 2r r}
 \pi_{2r,r}:=(-1)^r (2n-4r+1)!!\sum_{i=0}^r (-1)^{i} \frac{(2r-2i-1)!!}{(2n-2r-2i+1)!!}\,\tau_{2r,i}
\end{equation}
to be the unique primitive valuation of degree $2r$ whose expansion in terms of the Tasaki valuations has leading term $\tau_{2r,r}$, and define for $k\ge 2r$
\begin{align}
\label{def higher pi}\pi_{k,r}&:=L^{k-2r} \pi_{2r,r}\\
\label{pi in terms of tau}  
 &= (-1)^r(2n-4r+1)!!\sum_{i=0}^r (-1)^{i}\frac{(k-2i)!}{(2r-2i)!} \frac{(2r-2i-1)!!}{(2n-2r-2i+1)!!}\,\tau_{k,i}
\end{align}
by \eqref{L tau}.

For further use, we note that by Equation \eqref{eq_def_sigma},
\begin{align}\label{mu coeff}
\pi_{2r,r}\equiv(-1)^r \frac{(2n-4r+1)!! (2r-1)!!}{(2n-2r+1)!!}&\left(\mu_{2r,0} + \frac{2(2r-n-1)}{2r-1}\mu_{2r,1}\right) \\
\notag& \quad \mod \langle\mu_{2r,i}: i >1\rangle 
, \quad 2r\le n.
\end{align}

\begin{Proposition} \label{diagonalization} For each  $0\le k \le 2n$ the valuations $\pi_{k,r}, \ 0 \le r\le \frac {\min(k,2n-k)}{2}$ constitute a basis of $\valun_k$. Furthermore,
$$\pi_{k,r} \cdot \pi_{2n-k,s}=0, \quad r \ne s.$$
\end{Proposition}

\begin{proof}
The fact that these elements constitute a base of $\valun_k$ follows at
once from the Lefschetz decomposition of the $\sltwo$-representation
$\valun$. If $r \neq s$, say $r>s$, then 
\begin{equation}
\pi_{k,r} \cdot \pi_{2n-k,s}= L^{k-2r} \pi_{2r,r} \cdot L^{k-2s}
\pi_{2s,s}=C \cdot L^{2n-2r-2s} \pi_{2r,r} \cdot \pi_{2s,s}=0
\end{equation} 
since $ L^{2n-4r+1}
\pi_{2r,r}=0$. \end{proof}

\begin{Lemma} \label{lemma_magic_formula}
For $0 \leq 2r \leq k \leq 2n-2r$, 
\begin{equation}  \label{magic_formula}
\widehat{ \pi_{k,r}} = \frac{(k-2r)!}{(2n-2r-k)!} \pi_{2n-k,r}.
\end{equation}
\end{Lemma}

\proof We assume, as we may, that $k\le n$.
 By the Hard Lefschetz Theorem of
Alesker \cite{ale03b}, $\Lambda^{n-k}:\valun_{2n-k} \to \valun_n$ is
injective, so it is enough to show that $ \Lambda^{n-k}\widehat{ \pi_{k,r}} = \frac{(k-2r)!}{(2n-2r-k)!}\Lambda^{n-k} \pi_{2n-k,r}$. By \eqref{intertwine} and the fact that the Fourier transform acts trivially on $\valun_n$,  the left hand side is just $\pi_{n,r}$. On the other hand, the relation
\eqref{iterated_commutators} yields 
\begin{equation*}
\Lambda \pi_{l,r} =(l-2r)(2n-2r-l+1) \pi_{l-1,r},
\end{equation*}
which after iterating $n-k$ times gives
\begin{equation*}
\Lambda^{n-k} \pi_{2n-k,r} =\frac{(2n-2r-k)!} {(k-2r)!}\pi_{n,r},
\end{equation*}
as claimed.
\endproof

{\bf Remark.} Comparing the algebra of $\valun$ to the cohomology of K\"ahler
manifolds, \eqref{magic_formula} corresponds to the {\it magic
formula} relating primitive forms, the Lefschetz operator and the Hodge
star operator (\cite{huyb04}, Prop. 1.2.31). 

\subsection{Two loose ends}
\label{section_eigenvalues}
We tie up two loose ends from sections \ref{hermitian section} and \ref{tasaki section}.

\begin{Proposition}\label{correct constant} The constants $\gamma_k$ from Lemma \ref{mu = cf} are given by
$$
\mu_{k,0}= (-1)^{k+1}\frac{(2\pi)^k}{2\omega_k(k-1)!} f_k.
$$
\end{Proposition}

To this end we will make use of two lemmas.
We say that a valuation in $\Val^{U(n)}_k$ is {\bf anisotropic} if its Klain function vanishes on the isotropic $k$-Grassmannian $\Gr_{k,0}$. Thus the space of anisotropic valuations is spanned by the $\mu_{k,p}, p \ge 1$.

\begin{Lemma} \label{anisotropic_ideal}
The space of anisotropic valuations is an ideal in $\Val^{U(n)}$. 
\end{Lemma}

\begin{proof}[Proof of Lemma \ref{anisotropic_ideal}]
Let $\phi \in \Val^{U(n)}_k$ be anisotropic, and $\psi \in \Val^{U(n)}$ of degree $l$.  
By \cite{befu06}, section 1.2.2, we may write
\begin{displaymath}
\psi(K)=\int_{\AGr_{2n-l}(\mathbb{C}^n)} \chi(K \cap \bar E) d\mu(\bar E)
\end{displaymath}
with some smooth measure $\mu$ on the affine Grassmannian $\AGr_{2n-l}(\mathbb{C}^n)$, and the product $\phi \cdot \psi$ may be expressed
\begin{displaymath}
 \phi \cdot \psi (K)=\int_{\AGr_{2n-l}(\mathbb{C}^n)} \phi(K \cap\bar E) d\mu(\bar E).
\end{displaymath}

If $K$ is contained in an isotropic subspace, then the same trivially holds 
true for $K \cap \bar E$. Since $\phi$ is anisotropic, the integrand on the right hand side vanishes. 
It follows that $\phi \cdot \psi$ is anisotropic. 
\end{proof}

{\bf Remark.} In fact the ideal of anisotropic valuations equals the principal ideal $(u) =(\tau_{2,1}) =(\mu_{2,1})$.

\begin{Lemma} \label{universal_relation_for_s}
\begin{align}
\label{t^i} t^i & =\frac{i! \omega_i}{\pi^i} \mu_i \\
\label{s} s & =\frac{1}{\pi} \left(\mu_{2,1}+\frac12 \mu_{2,0} \right)\\
\label{u} u &= \frac{2}{\pi}\,\mu_{2,1}.
\end{align}
\end{Lemma}

\proof
Clearly $t^0=\chi=\mu_0$ and $t=\frac{2}{\pi}\mu_1$ by equations (46) and (48) of \cite{fu06}. The relation \eqref{t^i} now follows by induction using Equation \eqref{ell_on_mu} (cf. also
\cite{fu06}, Corollary 3.4).  

Theorem \ref{relations theorem} implies that $s=\frac12 t^2$ locally at $n=1$. This implies
  that the value of $s$ on a complex disc is $1$. Thus
  $s=\frac{1}{\pi} \left(\mu_{2,1}+a \mu_{2,0}\right)$ for some
  $a \in \mathbb{R}$. Meanwhile, $-st+\frac13 t^3 =f_3=0$ locally at $n=2$. Therefore $-s+\frac13 t^2$ is primitive in $\Val^{U(2)}$ with respect to
  the given $\mathfrak{sl}(2,\mathbb{R})$-representation. 
 Since $\mu_{3,0}=\mu_{3,2}=0$ locally at $n=2$, this implies that 
\begin{align*}
0=\pi L\left(-s+\frac13 t^2\right) & = L\left(-\mu_{2,1}-a\mu_{2,0}+\frac23
(\mu_{2,0}+\mu_{2,1})\right)\\
& = -\frac13 \mu_{3,1}+\left(\frac23-a\right) 2 \mu_{3,1}.
\end{align*} 
by \eqref{ell_on_mu}. Thus $a=\frac12$. 
\endproof

\begin{proof}[Proof of Proposition \ref{correct constant}]
By the 
recursion \eqref{recursion} we have 
\begin{align}
\notag k f_k & = -t(k-1) f_{k-1}-s(k-2)f_{k-2}\\
\label{fk} & = -t \gamma^{-1}_{k-1}(k-1) \mu_{k-1,0} - \frac{t^2}{4} \gamma^{-1}_{k-2}(k-2)
\mu_{k-2,0}-\frac{\gamma^{-1}_{k-2}(k-2)}{4} u  \mu_{k-2,0}.
\end{align}
Since $u=\frac{2}{\pi} \mu_{2,1}$ is anisotropic, the same holds true for $u \cdot \mu_{k-2,0}$ by Lemma \ref{anisotropic_ideal}. Comparing the coefficients of $\mu_{k,0}$ in \eqref{fk}, we obtain using \eqref{L tau}
\begin{displaymath}
k \gamma^{-1}_k  = -(k-1)\gamma^{-1}_{k-1} \frac{\omega_k}{\pi \omega_{k-1}}
k- \frac{\gamma^{-1}_{k-2}(k-2)}{4} \frac{\omega_k}{\pi^2\omega_{k-2}} k (k-1),
\end{displaymath} 
from which the Proposition follows by induction.
\end{proof}

The next loose end is 
\begin{proof}[Proof of \eqref{eq_mu_tu} from section \ref{tasaki section}] We proceed by induction on $q$. Since $\tau_{k,0} $ is the $k$th intrinsic volume $\mu_k$, the case $q=0$ is \eqref{t^i}.
For the inductive step we observe first that since \eqref{L tau} may be reformulated as
$$
\tau_{k+1,p}= \frac{\pi \omega_k}{(k-2p+1)\omega_{k+1}}t\cdot \tau_{k,p}  ,
$$
it is enough to prove the desired relation for $\tau_{2r,r}$. To accomplish this we compare the expressions \eqref{eq_fu_polynomial_tu} for the $f_k$ with
\begin{align*}
f_k &=  (-1)^{k+1}\frac{2\omega_k(k-1)!}{(2\pi)^k} \,\mu_{k,0}\\
&=  (-1)^{k+1}\frac{2\omega_k(k-1)!}{(2\pi)^k}\, \sum_{i=0}^{\lfloor k/2\rfloor} (-1)^{i}  \tau_{k,i},
\end{align*}
which follows from Corollary \ref{mu tasaki} and Proposition \ref{correct constant}. Taking $k=2r$ and equating the two expressions, \eqref{eq_mu_tu} follows from the inductive hypothesis.
\end{proof}

\begin{Corollary} \label{multiplication_u}
If $2p \leq k$ then
 \begin{align}\label{umu}
  u \cdot \mu_{k,p}&\equiv 
\frac{4(p+1)}{\pi(k+2)} (
(2p+1)\mu_{k+2,p+1}-2(p+2) \mu_{k+2,p+2}) \\
\notag& \mod\langle \mu_{k+2,i}: i > p+2\rangle.
 \end{align}
\end{Corollary}
\begin{proof} Since  $u \cdot \tau_{k,p} =\frac {2(2p+1)(2p+2)}{\pi(k+2)} \tau_{k+2,p+1}$ by \eqref{eq_mu_tu}, the desired relation \eqref{umu} may be computed from the relations \eqref{eq_def_sigma}, \eqref{eq_mu_in_terms_of_sigma} between the $\tau$ and the $\mu$.
\end{proof}

{\bf Remark.} The two sides of \eqref{umu} are in reality precisely equal, although we will not use this fact.

\subsection{The main computation}

\begin{Proposition}
For all $k \geq 2r$ 
\begin{equation} \label{product_primitives}
 (\pi_{k,r}, \widehat{\pi_{k,r}})=\frac{8^r\pi^n}{\omega_k\omega_{2n-k}}\binom{n}{2r} \frac{(k-2r)!(2n-4r)!} {(n-r)!(2n-2r-k)!} \frac{(2n-4r+1)!!}{(2n-2r+1)!!}
\end{equation}
\end{Proposition}
\begin{proof}
We show first that for $2r \leq n$, the value of the Poincar\'e pairing \eqref{pd pairing} of $u^r$ and $\widehat{\pi_{2r,r}}$ is
\begin{equation}\label{u pi pairing}
(u^r, \widehat{\pi_{2r,r}}) =\left(\frac 8 \pi\right)^r \frac{n!}{(n-2r)!}\frac{(2n-4r+1)!!}{(2n-2r+1)!!}
\end{equation}
This follows in turn from the relation
\begin{equation}\label{first}
u \cdot \widehat{\pi_{2r,r}} = \frac{8(2n-2r+3)(n-2r+1)(n-2r+2)}{\pi(2n-4r+3)(2n-4r+5)} \widehat{ \pi_{2r-2,r-1}},
\end{equation}
after $r$ iterations, since $\widehat{\pi_{0,0}} =\widehat {\mu_0} = \mu_{2n}$.
Both sides of \eqref{first} lie in the kernel of the map $L: \valun_{2n-2r+2} \to \valun_{2n-2r+3}$, which is one-dimensional. In order to fix the proportionality factor, it suffices to compare the coefficients of $\mu_{2n-2r+2,n-2r+2}$ on the two sides (note that locally  $\mu_{2n-2r+2,n-2r+1}=0$). It is straightforward to carry this out using \eqref{umu}
and \eqref{mu coeff}.

To prove \eqref{product_primitives} observe first that by \eqref{eq_mu_tu},
\begin{equation*}
\pi_{2r,r} \equiv \tau_{2r,r} = \frac{\pi^{2r}}{\omega_{2r}(2r)!} u^r = \frac{\pi^r r!}{(2r)!} u^r\ \mod t.
\end{equation*}
Since $t \cdot\widehat{\pi_{2r,r}} = const. \,  t \cdot\pi_{2n-2r,r} = 0$, 
the case $k=2r$ follows from \eqref{u pi pairing}, the definition \eqref{pi 2r r} of $\pi_{2r,r}$, and \eqref{eq_mu_tu}.
If $k> 2r$ we use Lemmas \ref{obvious lemma} and \ref{lemma_magic_formula} to compute
\begin{align*}
 \pi_{k,r} \cdot \widehat{\pi_{k,r}}& =\frac{(k-2r)!}{(2n-2r-k)!}\pi_{k,r}\cdot\pi_{2n-k,r}\\
&=\frac{(k-2r)!}{(2n-2r-k)!}\,(L^{k-2r}\pi_{2r,r})\cdot(L^{2n-k-2r}\pi_{2r,r})\\
&=\frac{(k-2r)!}{(2n-2r-k)!}\frac{\omega_{2r}\omega_{2n-2r}}{\omega_k\omega_{2n-k}}\,\pi_{2r,r}\cdot L^{2n-4r}\pi_{2r,r}\\
&=\frac{(k-2r)!(2n-4r)!}{(2n-2r-k)!}\frac{\omega_{2r}\omega_{2n-2r}}{\omega_k\omega_{2n-k}}\,\pi_{2r,r}\cdot\widehat{\pi_{2r,r}},\\
\end{align*}
which with the previous case yields \eqref{product_primitives}.
\end{proof}


Using Theorem \ref{g-kinematic}, the relation \eqref{product_primitives}, Proposition \ref{diagonalization} and Lemma \ref{lemma_magic_formula} now yield at once
\begin{Theorem}\label{pkf}
Set $p:=\min\left\{\lfloor \frac k 2\rfloor ,\lfloor \frac{2n-k}2\rfloor\right\}$. 
\begin{align}\label{pkf1} 
 k_{U(n)}(\chi)&=\frac 1{\pi^n}\sum_{k=0}^{2n}{\omega_k\omega_{2n-k}}\\
\notag & \quad\quad \sum_{r=0}^p \frac{(2n-2r-k)!(n-r)!}{8^r(k-2r)!(2n-4r)!} \frac{(2n-2r+1)!!}{(2n-4r+1)!!}\binom n {2r}^{-1} \pi_{k,r} \otimes \widehat{\pi_{k,r}}\\
\label{pkf2} &=\frac 1{\pi^n}\sum_{k=0}^{2n}{\omega_k\omega_{2n-k}}\\
\notag &\quad\quad\sum_{r=0}^p\frac{(n-r)!}{8^r(2n-4r)!} \frac{(2n-2r+1)!!}{(2n-4r+1)!!}\binom n {2r}^{-1}\pi_{k,r} \otimes{\pi_{2n-k,r}}
\end{align}
\end{Theorem}

\begin{Corollary} \label{+ def cor} The Tasaki matrices $T^n_k$, and the matrices $Q^n_k$ of \cite{fu06}, are positive definite.
\end{Corollary}
\begin{proof} These matrices are the inverses of those arising respectively by expressing the bilinear forms
$$(\phi ,\psi) \mapsto \phi \cdot \widehat \psi, \quad (\phi ,\psi) \mapsto  \phi \cdot t^{2n-k}\psi $$
on $\valun_k$ in terms of specific bases (the Tasaki valuations in the first case and the monomials in $s$ and $t$ in the second). Both of these diagonalize upon change of basis to the $\pi_{k,r}$, and the diagonal entries are the inverses of the (positive) coefficients of \eqref{pkf1} in the first case, and positive multiples of these in the second (by Lemma \ref{lemma_magic_formula} and the definition \eqref{def higher pi} of the $\pi_{k,r}$).
\end{proof}

Expanding via \eqref{pi in terms of tau} we obtain
\begin{Corollary}\label{tasaki entries} The $(i,j)$ entry of the Tasaki matrix $T^n_k$ is
\begin{align*}
\left(T^n_k\right)_{ij}& = (-1)^{i+j}\frac{\omega_k\omega_{2n-k}}{\pi^n }\\
&  \sum_{r=\max(i,j)}^{\lfloor \frac k 2\rfloor}
\left[\binom n{2r}^{-1}\frac{(2n-2r-k)!(n-r)!(k-2i)!(k-2j)!}{8^r(k-2r)!(2n-4r)!(2r-2i)!(2r-2j)!}\right.\\
&\quad\left.\times
\frac{(2n-2r+1)!! (2n-4r+1)!!(2r-2i-1)!!(2r-2j-1)!!}{(2n-2r-2i+1)!!(2n-2r-2j+1)!!}
\right].
\end{align*}
\end{Corollary}
We have not been able to simplify this expression further.
However, for fixed $k$, the above sum is finite and can be computed in a closed form. Thus it is straightforward (albeit messy) to calculate
\begin{equation}\label{t n 2}
T^n_2= \frac{1}{4n(n-1)}\left(\begin{matrix} 2n-1 & -1\\
-1 & 2n-1
\end{matrix}\right)
\end{equation}
\begin{equation}\label {t n 3}
T^n_3= \frac{2^{n-2}(n-3)!}{n\pi(2n-3)!!}\left(\begin{matrix} 2n-3 & -1\\
-1 & \frac{2n-1}{3}
\end{matrix}\right)
\end{equation}
\begin{align}\label{t n 4}
T^n_4&= \frac{(n-4)!}{16\, n!} \left(\begin{matrix} 3(2n-5)(2n-3) & -3(2n-3) & 9\\
-3(2n-3) & 2n^2-4n+3 & -3(2n-3)\\
9& -3(2n-3) & 3(2n-5)(2n-3)
\end{matrix}\right),
\end{align}
etc. The matrices $T^n_2, T^3_3$ had previously been computed in \cite{tasaki03} using the template method.

Note that since $k= 2,4$ are even, the matrices $T^n_2, T^n_4$ display both the expected diagonal symmetry and the antidiagonal symmetry predicted by Theorem \ref{palindrome thm}. In fact that theorem gives a family of identities among the values given in Corollary \ref{tasaki entries} whenever $k$ is even. From a practical perspective this is an aid in computing closed forms for these expressions, since for larger values of $i,j$ the sum in Corollary \ref{tasaki entries} is shorter.


\subsection{Other kinematic formulas}
\label{kinematic formulas}
Of course the whole point of the computations above is to give explicit forms for the kinematic formulas 
$$k_{U(n)}(\chi)(A,B) = \int_{\overline{U(n)} }\chi(A \cap \bar g B) \, d\bar g,$$ 
which in turn specialize to Crofton formulas when $A,B\subset \C^n$ are compact $C^1$ submanifolds (or even rectifiable sets) of complementary dimension. By the transfer principle (Theorem  \ref{transfer principle}), the latter formulas hold verbatim if $\C^n$ and $\overline{U(n)}$ are replaced by the spaces $\C P^n$ or $\C H^n$ of constant holomorphic sectional curvature together with their groups of isometries, with measures $d \bar g$ given by the standard convention \eqref{standard convention}. 

In the case of $\C P^n$, however, another natural convention is to take $d\bar g$ to be a probability measure. The resulting Crofton formulas may then be viewed as a generalization of B\'ezout's theorem. Normalizing the metric to be the standard Fubini-Study metric (i.e. with holomorphic sectional curvature $4$), they are obtained by dividing the constants above by $\vol_{2n}(\C P^n)= \frac{\pi^n}{n!}$. It is reassuring to recover B\'ezout's theorem for pairs (algebraic curve, algebraic hypersurface) and (algebraic surface, algebraic variety of codimension 2) from the matrices \eqref{t n 2}, \eqref{t n 4}, using the fact that for varieties $V^k, W^{n-k}\subset \C P^n$
\begin{align*}
\tau_{2k,q}(V) =\binom k q \mu_{2k,k}(V )&= \binom k q \mu_{2k}(V) =\binom k q  \frac{\pi^k}{k!}\deg(V), \\
\widehat{\tau_{2k,q}}(W) =\binom k q \mu_{2n-2k,n-k}(W )&= \binom k q \mu_{2n-2k}(W) =\binom k q  \frac{\pi^{n-k}}{(n-k)!}\deg(W),
\end{align*}
as may be computed via \eqref{eq_def_sigma}.

The calculations above also permit us to compute in explicit form the kinematic formulas $k_{U(n)}(\tau_{k,p})$, using the fundamental relation \eqref{general kf} and the product formula
\begin{equation}\label {tasaki product formula}
\tau_{k,p}\cdot \tau_{l,q} = \frac{\omega_{k+l}}{\omega_k \omega_l} \binom {k+l-2p-2q}{k-2p}\binom{2p+2q}{2p}\tau_{k+l,p+q},
\end{equation}
which is a simple consequence of \eqref{eq_mu_tu}. 
Rather than write down further messy general formulas, we illustrate by
computing the expected value of the length of the curve given by the intersection of a real 4-fold and a real 5-fold in $\C P^4$.
\begin{Theorem} Let $M^4, N^5 \subset \C P^4$ be real $C^1$ submanifolds of dimension $4,5$ respectively. Let $\theta_1,\theta_2$ be the K\"ahler angles of the tangent plane to $M$ at a general point $x$ and $\psi$ the K\"ahler angle of the orthogonal complement to the tangent plane to $N$ at $y$. Let $dg$ denote the invariant probability measure on $U(5)$. Then
\begin{align*}
\int_{U(5)}\length &({M \cap g N} )\, dg= \\
\left.\frac{1}{5\pi^4} \times\right[&
 30\vol_4(M)\vol_5(N) 
 -6 \vol_4(M)\int_N \cos^2\psi\, dy \\
&\quad-3\int_M (\cos^2\theta_1+\cos^2\theta_2 )\, dx \cdot \vol_5(N)\\
&\quad\left.+7\int_M (\cos^2\theta_1+\cos^2\theta_2)\, dx\cdot \int_N \cos^2\psi\, dy \right].
\end{align*}
\end{Theorem}
\begin{proof} If $l \subset \C P^n$ is a real curve, then $\tau_{1,0}(l) = \length(l)$. Thus by the transfer principle we wish to compute the terms of bidegree $(4,5)$ in $\frac{4!}{\pi^4}k_{U(4)}(\tau_{1,0})$. Since
$$ \tau_{4,0}\cdot\tau_{1,0}= \frac83 \tau_{5,0},\quad
\tau_{4,1}\cdot\tau_{1,0}= \frac85 \tau_{5,1}, \quad
\tau_{4,2}\cdot\tau_{1,0}= \frac{8}{15} \tau_{5,2}
$$
the matrix giving the relevant terms is 
\begin{equation*}
\frac{1}{10 \pi^4} \left(\begin{matrix} 75 & -15 & 3\\ -25 & 19 &  -5\\                          
15 & -15 & 15
  \end{matrix}\right)
\end{equation*}
where the columns are indexed by the $\tau_{5,i}$ and the rows by the $\widehat{\tau_{4,j}} =\tau_{4,j}$.
Locally at $n=4$,
$$
\tau_{5,0}= \widehat{\tau_{3,0}},\quad
\tau_{5,1}= \widehat{\tau_{3,0}}+ \widehat{\tau_{3,1}},\quad
\tau_{5,2}= \widehat{\tau_{3,1}}
$$ 
so with respect to the bases $ \widehat{\tau_{3,i}}, \tau_{4,j}$ one computes the pairing matrix to be
\begin{equation}\label{length pairing matrix}
\frac{1}{5\,\pi^4} \left(\begin{matrix} 30 & -6\\
-3 & 7\\
0 & 0
\end{matrix}\right).
\end{equation}
\end{proof}

By \cite{befu06}, we can also translate this result to give an additive kinematic formula for the average 7-dimensional volume of the Minkowski sum of two convex subsets in $\C^4$ of dimensions 3 and 4 respectively. 
\begin{Theorem} Let $E\in \Gr_4(\C^4), F\in \Gr_3(\C^4)$; let $\theta_1,\theta_2$ be the K\"ahler angles of $E$ and $\psi$ the K\"ahler angle of $F$. Let $dg$ be the invariant probability measure on $U(4)$. If $A\in \K(E), B\in \K(F)$ then
\begin{align*}\int_{U(4)}\vol_7& (A+ gB)\, dg=\frac{1}{120} \vol_4(A)\vol_3(B) \times \\
&\left[30 -6 \cos^2\psi -3(\cos^2\theta_1+\cos^2\theta_2) + 7\cos^2\psi (\cos^2\theta_1+\cos^2\theta_2)\right].
\end{align*}
\end{Theorem}
\begin{proof} Recall that the additive kinematic operator $a_{U(4)}: \Val^{U(4)}(\C^4)\to \Val^{U(4)}(\C^4)\otimes\Val^{U(4)}(\C^4)$ is given by
$$
a_{U(4)}(\phi) (A,B) = \int_{U(4)} \phi(A+gB)\, dg.
$$
By Theorem 1.7 of \cite{befu06}, 
$$a_{U(4)}(\mu_7)= \widehat{ k_{U(4)}( \widehat{\mu_7}) } =\widehat{ k_{U(4)}( {\mu_1}) }.
$$
Thus the bidegree $(3,4)$ terms of 
$a_{U(4)}(\mu_7)$ are given with respect to the bases $\tau_{3,i}, \tau_{4,j}$ by the matrix $\frac{\pi^4}{4!}\times$  \eqref{length pairing matrix}.
\end{proof}


\begin{thebibliography}{99}
\bibitem{ale03b} Alesker, S.: Hard Lefschetz theorem for valuations, complex integral geometry, and unitarily invariant
  valuations. {\it J. Differential Geom.} \textbf{63} 
  (2003), 63--95. 

\bibitem{ale04} Alesker, S.: The multiplicative structure on
  polynomial valuations. {\it Geom. Funct. Anal.} \textbf{14} (2004),
  1--26.

\bibitem{ale04a} Alesker, S.: Hard Lefschetz theorem for valuations
  and related questions of integral geometry. {\it Geometric aspects
  of functional analysis}, 9--20, LNM 1850, Springer, Berlin 2004. 
  
\bibitem{ale-be} Alesker, S., Bernstein, J.: Range characterization of the cosine transform on higher Grassmannians, {\it Adv. Math.} \textbf{184} (2004), 367--379.


\bibitem{ale05a} Alesker, S.: Theory of valuations on manifolds
  I. Linear spaces. {\it Israel J. Math.} \textbf{156} (2006), 311--339.

\bibitem{ale05b} Alesker, S.: Theory of valuations on manifolds
  II. {\it Adv. Math.} \textbf{207} (2006), 420--454. 

\bibitem{ale05d} Alesker, S.: Theory of valuations on manifolds
  IV. New properties of the multiplicative structure. In {\it Geometric aspects of functional analysis}, LNM \textbf{1910} (2007), 1--44.

\bibitem{ale06} Alesker, S.: Valuations on manifolds: a survey. {\it Geom. Funct. Anal.} \textbf{17} (2007), 1321--1341.

\bibitem{alefu05} Alesker, S., Fu, J. H. G.: Theory of valuations on manifolds
  III.  Multiplicative structure in the general case. {\it Trans. Amer. Math. Soc.}  \textbf{360} (2008), 1951--1981.

\bibitem{ale07} Alesker, S.: A Fourier type transform on translation invariant valuations on convex sets. Preprint arXiv:math/0702842.

\bibitem{be07} Bernig, A.: A Hadwiger-type theorem for the special unitary group. {\it Geom. Funct. Anal.} \textbf{19} (2009), 356--372.

\bibitem{bebr07} Bernig, A., Br\"ocker, L.: Valuations on manifolds
  and Rumin cohomology. {\it J. Differential Geom.} \textbf{75} (2007), 433--457.

\bibitem{befu06} Bernig, A., Fu. J. H. G.: Convolution of convex valuations.  {\it Geom. Dedicata} \textbf{123} (2006), 153--169. 

\bibitem{chern66} Chern, S.S.: On the kinematic formula in integral geometry. {\it J. of Math. and Mech.} \textbf{16} (1966), 101--118.

\bibitem{del80} Deligne, P.: La conjecture de Weil. II. {\it Publication IHES} \textbf{52} (1980), 137--252. 
  
\bibitem{federer59} Federer, H.: Curvature measures. {\it Trans. Amer. Math. Soc.} \textbf{93} 1959, 418--491. 
 
\bibitem{fu90} Fu, J. H. G.: Kinematic formulas in integral geometry. {\it Indiana Univ. Math. J.} \textbf{39} (1990), 1115--1154.

\bibitem{fu94} Fu, J. H. G.: Curvature measures of subanalytic sets. {\it Amer. J. Math.} \textbf{116} (1994), pp. 819--880.

\bibitem{fu06} Fu, J. H. G.: Structure of the unitary valuation
  algebra. {\it J. Differential Geom.}
  \textbf{72} (2006), 509--533. 

\bibitem{fu06b} Fu, J. H. G.: Integral geometry and Alesker's theory of valuations, in {\it Integral Geometry and Convexity: Proceedings of the International Conference, Wuhan, China, 18 - 23 October 2004 } (Eric L. Grinberg, Shougui Li, Gaoyong Zhang, Jiazu Zhou, eds.)
 World Scientific, Singapore 2006, pp. 17--28.

\bibitem{grha78} Griffiths, P., Harris, J.: {\it Principles of
    Algebraic Geometry}. Wiley, New York 1978. 

\bibitem{hadwiger57} Hadwiger, H.: {\it Vorlesungen \"uber Inhalt, Oberfl\"ache und Isoperimetrie}. Springer, Berlin, 1957.

\bibitem{howard93} 
Howard, R.: The kinematic formula in Riemannian homogeneous spaces. {\it Mem. Amer. Math. Soc.} \textbf{509} (1993).

\bibitem{huyb04} Huybrechts, D.: {\it Complex Geometry.} Springer
  Universitext, Berlin, 2005. 

\bibitem{kata02} Kang, H. J., Tasaki, H.: Integral geometry of real surfaces in the complex projective plane. {\it Geom. Dedicata} \textbf{90} (2002), 99--106. 

\bibitem{kl00} Klain, D.: Even valuations on convex bodies. {\it
    Trans. Amer. Math. Soc.} \textbf{352} (2000), 71--93. 

\bibitem{klro97} Klain, D., Rota, G.-C.: {\it Introduction to Geometric
  Probability}. Lezione Lincee, Cambridge University Press 1997. 
    
\bibitem{nijenhuis74} Nijenhuis, A.: On Chern's kinematic formula in integral geometry. {\it J. Differential Geom.} \textbf{9} (1974), 475--482.

\bibitem{pa02} Park, H.: Kinematic formulas for the real subspaces of
  complex space forms of dimension $2$ and $3$. PhD-thesis University
  of Georgia 2002. 

\bibitem{rumin} Rumin, M.: Formes diff\'erentielles sur les vari\'et\'es de contact.  {\it J. Differential Geom.} \textbf{39} (1994), 281--330.  

\bibitem{sa} Santal\'o, L.A.: {\it Integral geometry and geometric
probability}. Cambridge University Press 1978. 

\bibitem{schn93} Schneider, R.: {\it Convex bodies: the Brunn-Minkowski theory}. Cambridge Univ. Press 1993. 

\bibitem{tasaki00} Tasaki, H.: Generalization of K\"ahler angle and
  integral geometry in complex projective spaces. in: Steps in differential geometry. Proceedings of the colloquium on differential geometry Debrecen, 349-361 (2001). (www.emis.de\slash proceedings\slash CDGD2000\slash pdf\slash K\underline{ }Tasaki.pdf )

\bibitem{tasaki03} Tasaki, H.: Generalization of K\"ahler angle and integral geometry in complex projective spaces II. {\it Math. Nachr.} \textbf{252} (2003), 106--112. 

\bibitem{yaish73} Yano, K., Ishihara, S.: {\it Tangent and cotangent
    bundles.} Marcel Dekker, New York 1973. 
    
\end{thebibliography}
\end{document}